\documentclass[a4paper, 10pt]{article}

\usepackage[latin1]{inputenc}
\usepackage[T1]{fontenc}
\usepackage{lmodern}
\usepackage{hyperref}	
\usepackage[english]{babel} 
\usepackage{textcomp}
\usepackage{amsmath,amssymb,amsthm}
\usepackage{graphicx}
\usepackage{xcolor}
\usepackage{enumitem}	
\usepackage[babel=true]{csquotes} 
\usepackage{bm}  
\usepackage{algorithm, algpseudocode}
\usepackage{caption}

\usepackage{tikz}
\usetikzlibrary{arrows}
\usepackage{tikz-cd}
\usetikzlibrary{matrix,backgrounds}
\pgfdeclarelayer{myback}
\usepackage{filecontents,pgfplots}
\usepackage{graphicx}

\usepackage{geometry}
\usepackage{array}

\title{Computing Jacobi's \texorpdfstring{$\theta$}{theta-fuction} in quasi-linear time}
\author{Hugo Labrande$ ^{1,2,3,4}$ \\ \textbf{ } \\ \small	$^1$ Université de Lorraine, LORIA, UMR 7503, Vandoeuvre-lès-Nancy, F-54506, France \normalsize \\ \small	$^2$ LORIA / INRIA Lorraine, CARAMEL team, \normalsize \\ \small 615 rue du jardin botanique, 54602 Villers-lès-Nancy Cedex, France \normalsize \\  \small $^3$ CNRS, LORIA, UMR 7503, Vandoeuvre-lès-Nancy, F-54506, France \normalsize \\ \small	$^4$ University of Calgary, 2500 University Dr NW, Calgary, AB, Canada T2N1N4 \normalsize \\  \href{mailto:hugo.labrande@inria.fr}{hugo.labrande@inria.fr} 		}
\date{}

\theoremstyle{plain}
\newtheorem{theorem}{Theorem}[section]
\newtheorem{proposition}[theorem]{Proposition}
\newtheorem{lemma}[theorem]{Lemma}

\newtheorem{definition}[theorem]{Definition}

\renewcommand{\Re}{\operatorname{Re}}
\renewcommand{\Im}{\operatorname{Im}}
\DeclareMathOperator{\AGM}{AGM}

\DeclareMathOperator{\Value}{0.345}

\begin{document}

\maketitle

\begin{abstract}
Jacobi's $\theta$ function has numerous applications in mathematics and computer science; a naive algorithm allows the computation of $\theta(z,\tau)$, for $z, \tau$ verifying certain conditions, with precision $P$ in $O(\mathcal{M}(P) \sqrt{P})$ bit operations, where $\mathcal{M}(P)$ denotes the number of operations needed to multiply two complex $P$-bit numbers. We generalize an algorithm which computes specific values of the $\theta$ function (the \textit{theta-constants}) in asymptotically faster time; this gives us an algorithm to compute $\theta(z, \tau)$ with precision $P$ in $O(\mathcal{M}(P) \log P)$ bit operations, for any $\tau \in \mathcal{F}$ and $z$ reduced using the quasi-periodicity of $\theta$.
\end{abstract}

\section{Introduction}

Jacobi's $\theta$ function appears in a wide range of fields, such as non-linear differential equations (as a solution of the heat equation), the study of modular forms, and number theory, in which it is the main ingredient to convert between algebraic and analytic representations of elliptic curves. Namely, we have the embedding \cite[I.4]{Mumford1}
\begin{eqnarray*}
\mathbb{C}/(\mathbb{Z} + \tau \mathbb{Z}) & \rightarrow & \mathbb{P}^3(\mathbb{C}) \\
z & \mapsto & \left( \theta_{00}(2z, \tau), \theta_{01}(2z, \tau), \theta_{10}(2z, \tau), \theta_{11}(2z, \tau) \right)
\end{eqnarray*}
where the $\theta_i$ are essentially the $\theta$ function with its $z$ argument translated. We also have the equation
\begin{equation*}
\wp(z, \tau) = \frac{\pi^2}{3} (\theta_{10}^4(0,\tau) - \theta_{01}^4(0,\tau)) - \pi^2 \theta_{01}^2(0,\tau) \theta_{10}^2(0,\tau) \frac{\theta_{00}^2(z,\tau)}{\theta_{11}^2(z,\tau)}
\end{equation*}
which allows to compute, for any point on the torus, its $x$-coordinate on the curve $E(\mathbb{C}) : y^2 = 4x^3 - g_2 x - g_3$.

Special values of the $\theta$ function have interesting additional properties: the \textit{theta-constants}, the value of $\theta$ at points $z = 0, \frac{1}{2}$ and $\frac{\tau}{2}$. As modular forms in $\tau$, they are linked to other modular functions, such as the $j$-invariant or Dedekind's $\eta$ function. Computing the value of the theta-constants allows one to compute the value of those modular functions; this has been used in record computations of class polynomials~\cite{Enge}, which are interesting to generate safe cryptographic curves with the CM method.

The main problem we are dealing with here is to compute $\theta(z, \tau)$ with given absolute precision $P$, which allows us to compute the above embedding at any given precision. We will suppose throughout the article that $(z, \tau)$ satisfy certain conditions; the general case can be deduced from this one using formulas we mention later. The $\theta$ function is defined by a rapidly convergent series; under the conditions specified on $z, \tau$, it gives a naive algorithm that requires a running time of $O(\mathcal{M}(P) \sqrt{P})$ bit operations, where $\mathcal{M}(P)$ is the number of operations needed to multiply two $P$-bit complex numbers. Although fast, this is a worse running time than other transcendental functions such as the exponential of a complex number, which can be computed in \textit{quasi-optimal time} $O(\mathcal{M}(P) \log P)$.

There is an algorithm to compute the theta-constants asymptotically faster than with the naive method, outlined in~\cite{Dupont}. This algorithm relies on connections between theta-constants and the \textit{arithmetico-geometric mean} (AGM) of Gauss; the complex-valued AGM, when evaluated at the theta-constants, has interesting properties, and this is used along with Newton's method to, in a sense, invert the AGM and recover the values of the theta-constants. This algorithm allows computation of theta-constants for $\tau \in \mathcal{F}$ with precision $P$ in quasi-optimal time $O(\mathcal{M}(P) \log P)$, independently of $\tau$. It is faster than the naive method for precisions as low as a few thousand bits.

In this article, we provide a generalization of this algorithm which computes $\theta(z, \tau)$, for $\tau \in \mathcal{F}$ and $z$ such that $\Im(z) \leq \Im(\tau)/2$, with absolute precision $P$ in $O(\mathcal{M}(P) \log P)$ bit operations. We give two algorithms: the first one runs in quasi-optimal time in $P$, but its running time depends on $z$ and $\tau$; we then use this algorithm as a subroutine to build a quasi-optimal algorithm with complexity independent of $z$ and $\tau$, provided $\tau \in \mathcal{F}$ and $\Im(z) \leq \Im(\tau)/2$. An GNU MPC~\cite{MPC} implementation of the algorithm was realized; it is faster than the naive method for values of $P$ greater than a few hundred thousand digits.

Our algorithm provides the six values $\theta(z, \tau), \theta(z+\frac{1}{2}, \tau), \theta(z+\frac{\tau}{2}, \tau), \theta(0, \tau), \theta(\frac{1}{2})$ and $\theta(\frac{\tau}{2})$, which are sufficient to compute the projective embedding mentioned above. It can also be used to compute the Weierstrass $\wp$ function and its derivative in quasi-optimal time; hence, this paper provides a quasi-optimal time algorithm to compute the \enquote{Jacobi map} $\mathbb{C}/\Lambda \rightarrow E(\mathbb{C})$ of an elliptic curve. We note that the \enquote{Abel map} $E(\mathbb{C}) \rightarrow \mathbb{C}/\Lambda$ can already be computed in quasi-optimal time using links to elliptic integrals and the Landen isogeny; see \cite{BostMestre88} and \cite{CremonaThongjunthug}.

We note that~\cite{LutherOttenTheta} gives an algorithm to compute $\theta$ with real arguments (i.e. $\theta(u, m)$ with $0 < m < 1$ and $0 \leq u \leq K(m)$), defined by its representation as an infinite product, with the same, quasi-optimal complexity. Their algorithm relies on the Landen transform for the $\theta$ function, and could perhaps be generalized to the complex setting. We had independently pursued this line of thought, but found that in the complex setting, the presence of trigonometric functions induced heavy precision losses for some inputs; however, there may be a workaround those issues, which would allow one to find an algorithm for the complex setting and with quasi-optimal complexity relying on the Landen transform.

This article is organized as follows. We introduce the necessary mathematical background and the strategies we follow for argument reduction in Section~2, which justifies our choice to consider throughout the paper the case $\tau \in \mathcal{F}$ and $\Im(z) \leq \Im(\tau)/2$; we then provide an analysis of the naive algorithm for $\theta(z,\tau)$ under those conditions. Section~3 introduces a sequence derived from relations between values of $\theta$, and we prove quadratic convergence of a certain homogeneization of the sequence; this is what replaces the AGM in the general case. Section~4 gives a first algorithm for computing $\theta$-functions, with a complexity depending on $z$ and $\tau$; this is quasi-optimal provided that $z,\tau$ belong to a compact set. We then get rid of the dependency in $z, \tau$ much in the same way as in the case of theta-constants, which gives a uniform algorithm with complexity $O(\mathcal{M}(P) \log P)$. Section~5 shows timings for our GNU MPC implementation of this last algorithm and compare it to our implementation of the naive algorithm.

\section{\texorpdfstring{The function $\theta$, and
$\theta$-constants}{The function theta, and theta-constants}}

\subsection{Definitions and argument reduction}
\label{section:argumentreduction}

We recall a few basic facts, following the presentation of~\cite{Mumford1}.
\begin{definition}
Define, for $z \in \mathbb{C}$ and $\tau \in \mathcal{H}$ (i.e. $\Im \tau > 0$)
\begin{equation*}
\theta(z, \tau) = \sum_{n \in \mathbb{Z}} \exp\left(\pi i \tau n^2 + 2 \pi i n z\right).
\end{equation*}
\end{definition}
\begin{proposition}[quasi-periodicity]
We have $\theta(z + 1, \tau) = \theta(z, \tau)$ and $\theta(z + \tau, \tau) = e^{- \pi i z - 2 \pi i \tau} \theta(z, \tau)$; in fact for any integers $a,b$,
\begin{equation}
\theta(z + a\tau + b, \tau) = e^{- i \pi a^2 \tau - 2 i \pi a z} \theta(z,\tau).
\label{eq:quasiperiodicitytheta}
\end{equation}
\end{proposition}

We also define the following variants of the theta function (which are related to the definition of ``theta functions with characteristics'')~\cite[Section~I.3]{Mumford1}:
\begin{definition} \label{defineFlavorsofTheta}
\begin{eqnarray*}
\theta_{00}(z, \tau) = \theta\left(z, \tau\right) & \qquad & \theta_{10}(z, \tau) = \exp\left(\pi i \tau/4 + \pi i z \right) \theta\left(z + \frac{\tau}{2}, \tau\right) \\
\theta_{01}(z, \tau) = \theta\left(z+\frac{1}{2}, \tau\right) & \qquad & \theta_{11}(z, \tau) = \exp\left(\pi i \tau/4 + \pi i (z+1/2) \right) \theta\left(z + \frac{\tau+1}{2}, \tau\right) \\
\end{eqnarray*}
We define \textbf{theta-constants} as the values in 0 of those functions.
\end{definition}
Those functions and their theta-constants are linked by a great number of formulas; we will give such formulas as we use them, and most of them can be found in \cite[Section~I.5]{Mumford1}. Note that we have:
\begin{proposition}
For any $\tau$, the functions $z \mapsto \theta_{00}(z,\tau), z \mapsto \theta_{01}(z,\tau), z \mapsto \theta_{10}(z,\tau)$ are even, while $z \mapsto \theta_{11}(z,\tau)$ is odd.
\end{proposition}
The latter implies that $\theta_{11}(0,\tau) = 0$, so the only theta constants we are interested in are $\theta_{00}(0,\tau), \theta_{01}(0,\tau)$, and $\theta_{10}(0,\tau)$.

Those properties can be used for the purpose of argument reduction. For instance, we can use the parity of $\theta$ to suppose that $\Im(z) \geq 0$; if this is not the case, one can consider $-z$ instead of $z$, which does not change the value of $\theta$ but ensures that $\Im(z) \geq 0$. Furthermore, Equation~\eqref{eq:quasiperiodicitytheta} can be used to recover $\theta(z,\tau)$ from the value of $\theta(z',\tau)$, where $z'$ is such that $|\Re(z')| \leq \frac{1}{2}$ and $|\Im(z')|~\leq~\frac{\Im(\tau)}{2}$; the added cost is the cost of computing an exponential factor. This exponential factor can become quite big; should one want to compute $\theta(z,\tau)$ with an error of at most $2^{-P}$, they have to work with representations of at least $P + C$ bits, with
\begin{equation*}
C = \log_2(|\theta(z',\tau)|) + \pi \log_2(e) (a^2 \Im(\tau) + 2a \Im(z)) + 2
\end{equation*}
This is because the integral part of the result fits in $C$ bits, while the fractional part should be coded on at least $P$ bits to ensure a final error bounded by $2^{-P}$. The complexity of running our algorithm and computing the exponential factor will be $O(\mathcal{M}(P + C) \log(P + C))$, and hence will depend on $\tau$ and $z$; this is inevitable. Hence, throughout the paper we suppose that $z$ is reduced, in the sense that $|\Re(z)| \leq \frac{1}{2}$ and $0 \leq \Im(z) \leq \frac{\Im(\tau)}{2}$, with the understanding that the step of argument reduction has a complexity depending on the original values of $z$ and $\tau$. However, as Section~\ref{subsection:naive} shows, this hypothesis, combined with a hypothesis in $\tau$, allows us to work with values of $\theta$ bounded by 4, which allows us to write an algorithm with complexity only depending on $P$ for any $z$ satisfying these conditions.

We can also reduce the second argument of $\theta$. Define the action of $SL_2(\mathbb{Z})$ on the complex upper-half plane $\mathcal{H}$ by
\begin{equation*}
\begin{pmatrix} a & b \\ c & d \end{pmatrix} \cdot \tau \mapsto \frac{a \tau + b}{c \tau + d}.
\end{equation*}
Its fundamental domain is 
\begin{equation*}
\mathcal{F} = \{ \omega \in \mathcal{H} \enskip | \enskip |\Re(\omega)| < 1/2, |\omega|>1 \}
\end{equation*}
Computing $\tau' \in \mathcal{F}$ and $M = \begin{pmatrix} a & b \\ c & d \end{pmatrix} \in SL_2(\mathbb{Z})$ such that $\tau' = M \tau$ can be done by finding the shortest vector in the lattice $(1, \tau)$ using Gauss's algorithm~\cite{Vallee}, which (since the inertia is small) will be asymptotically negligible. The value of $\theta(z,\tau)$ can then be computed from $\theta(z', \tau')$ (for some value $z'$) using the following theorem:
\begin{theorem}[extension of {\protect\cite[Theorem~7.1]{Mumford1}}]
Let $\tau \in \mathcal{H}$ and $z \in \mathbb{C}$, and let \\$\gamma = \begin{pmatrix}
a & b \\ c & d
\end{pmatrix} \in SL_2(\mathbb{Z})$. Suppose $c > 0$, or $c = 0$ and $d > 0$; if not, take $-\gamma$. Then we have:
\begin{equation}
\label{eq:sl2andthetavalues}
\theta_i\left(\frac{z}{c\tau +d}, \frac{a \tau +b}{c \tau +d}\right) = \zeta_{i, \gamma, \tau} \sqrt{c \tau + d} e^{i \pi c z^2/(c \tau +d)} \theta_{\sigma(i)}(z, \tau)
\end{equation}
where the square root is taken with positive real part, $\zeta_{i, \gamma, \tau}$ is an eighth root of unity and $\sigma$ is a permutation of the elements $(00,01,10)$, defined by the following table:
\begin{table}[H]
\footnotesize
\begin{center}
\begin{tabular}{|c|c|c|c||c|}
\hline
a & b & c & d & $\sigma(00,01,10)$ \\
\hline
odd & even & even & odd & $(00,01,10)$ \\
odd & odd & even & odd & $(01,00,10)$ \\
odd & even & odd & odd & $(10,01,00)$ \\
even & odd & odd & even & $(00,10,01)$ \\
odd & odd & odd & even & $(10,00,01)$ \\
even & odd & odd & odd & $(01,10,00)$ \\
\hline
\end{tabular}
\end{center}
\normalsize
\end{table}
\end{theorem}
\begin{proof}
Define for any $\gamma = \begin{pmatrix}
a & b \\ c & d
\end{pmatrix} \in SL_2(\mathbb{Z})$:
\begin{equation*}
e_{\gamma}(\tau) = c \tau + d, \qquad f_{\gamma}(z, \tau) = e^{i \pi c z^2/(c \tau + d)}
\end{equation*}
With a bit of care, one can prove
\begin{equation*}
e_{\gamma_1}(\gamma_2 \tau) e_{\gamma_2}(\tau) = e_{\gamma_1 \gamma_2}(\tau),  \qquad f_{\gamma_1}\left(\frac{z}{c_2 \tau + d_2}, \gamma_2 \tau\right) f_{\gamma_2}(z, \tau) = f_{\gamma_1 \gamma_2}(z, \tau)
\end{equation*}
The maps $T: \tau \mapsto \tau + 1$ and $S: \tau \mapsto \frac{-1}{\tau}$ are generators of $SL_2(\mathbb{Z})$, and we have \cite[Table V, p. 36]{Mumford1}
\begin{equation*}
\theta_i\left(z, T \tau\right) = \zeta_{i,T, \tau} \sqrt{e_{T}(\tau)} f_{T}(z, \tau) \theta_{\sigma_T(i)}(z, \tau), \qquad \theta_i\left(\frac{z}{\tau}, S \tau\right) = \zeta_{i,S, \tau} \sqrt{e_{S}(\tau)} f_{S}(z, \tau) \theta_{\sigma_S(i)}(z, \tau)
\end{equation*}
with square roots taken with real parts and for some $\zeta \in \mathbb{U}_8$ and $\sigma_S, \sigma_T \in \mathfrak{S}_3$. Hence for all $\gamma \in SL_2(\mathbb{Z})$,
\begin{equation*}
\theta_i\left( \frac{z}{e_{\gamma}(\tau)^2}, \gamma \tau \right) = \zeta_{i, \gamma, \tau} \sqrt{e_{\gamma}(\tau)} f_{\gamma}(z, \tau) \theta_{\sigma_{\gamma}(i)}(z, \tau)
\end{equation*}
for some root of unity and some permutation. The correspondance $\gamma \mapsto \sigma_{\gamma}$ can be determined from \cite[p.~36]{Mumford1}, although one can simply notice it is independent of $z$ and use the tables found by Gauss in the case of theta-constants \cite[Eq.~2.15]{Cox}. Finally, we could attempt to give a formula for $\zeta_{i, \gamma, \tau}$, but it is more efficient to simply compute a very low precision approximation of $\theta(z, \tau)$ and compare it to the full-precision value to determine which eighth root is needed.



\end{proof}

Thus, in order to recover $\theta_i(z, \tau)$ from $\theta_i\left(\frac{z}{c \tau + d}, \frac{a \tau + b}{c \tau + d} \right)$, one needs to compute $\sqrt{c \tau + d}$ (which is done in $O(\mathcal{M}(P))$ bit operations) and $e^{\pi i c z^2 / (c\tau+d)}$ (done in $O(\mathcal{M}(P) \log P)$ bit operations), and perform a division; determining $\zeta$ is asymptotically negligible. The cost of this step is then $O(\mathcal{M}(P) \log P)$ bit operations.

We note that in general, because of the permutation $\sigma_{\gamma}$, we need to have computed all three values $\theta_{00,01,10}\left(\frac{z}{c \tau + d}, \frac{a \tau + b}{c \tau + d} \right)$ in order to be able to use the formula to compute, say, $\theta_{00}(z,\tau)$. We will occasionally talk about computing $\theta_{11}$, but this will not be the focus of the paper. Hence, the problem we consider in this paper is the following:
\begin{center}
\vspace*{-0.6cm}
\begin{eqnarray}
\text{Compute } && \theta_{00}(z,\tau), \theta_{00}(0,\tau),  \theta_{01}(z,\tau), \theta_{01}(0,\tau),  \theta_{10}(z,\tau), \theta_{10}(0,\tau) \notag \\
\text{where } && |\tau| > 1, \quad |\Re(\tau)| \leq \frac{1}{2}, \quad \Im(\tau) > 0, \notag \\
&&|\Re(z)| \leq \frac{1}{2}, \quad 0 \leq \Im(z) \leq \frac{\Im(\tau)}{2} \label{eq:conditions} \\
\text{ in quasi-optimal time,} && \text{ i.e. } O(\mathcal{M}(P) \log P). \notag
\end{eqnarray}
\end{center}

\subsection{\texorpdfstring
    {Naive algorithm to compute $\theta$}
    {Naive algorithm to compute theta}
}
\label{subsection:naive}

\subsubsection{\texorpdfstring
    {Partial summation of the series defining $\theta$}
    {Partial summation of the series defining theta}
}

We define the following partial summation for the series defining $\theta(z, \tau)$:
\begin{equation*}
S_B(z, \tau) = 1 + \sum_{0 < n < B} q^{n^2} ( e^{2 i \pi n z} + e^{-2 i \pi n z})
\end{equation*}
where use the notation $q = e^{i \pi \tau}$. We have
\begin{proposition}
\label{naiveboundtheta}
Suppose that $\Im(\tau) \geq 0.35$ and $0 \leq \Im(z) \leq \Im(\tau)/2$; in particular, this is the case if the conditions~\eqref{eq:conditions} are satisfied. Then, for $B \geq 1$, $|\theta(z,\tau) - S_B(z,\tau)| \leq 3 |q|^{(B-1)^2}$.
\end{proposition}
\begin{proof}
We look at the remainder of the series:
\begin{eqnarray} \label{eq:proofnaive}
|\theta(z, \tau) - S_B(z, \tau)| & \leq & \sum_{n \geq B} |q|^{n^2} (|e^{2 i \pi n z}| + |e^{-2 i \pi n z}|) \notag \\
& \leq & \sum_{n \geq B} |q|^{n^2} (1+|q|^{-n}) \leq 2 \sum_{n \geq B} |q|^{n^2-n} \notag \\
& \leq & 2 \sum_{n \geq B} |q|^{(n-1)^2} \leq 2 \sum_{n \geq 0} |q|^{(B-1+n)^2} \notag \\
& \leq & 2 |q|^{(B-1)^2} \sum_{n \geq 0} |q|^{2n(B-1)+n^2} \notag \\
& \leq & 2 \frac{|q|^{(B-1)^2}}{1 - |q|^{2B-1}}
\end{eqnarray}
A numerical calculation shows that for $\Im(\tau) \geq 0.35$, we have $\frac{2}{1-|q|} \leq 3$, which proves the proposition.
\end{proof}
We can prove the same inequality for $\theta_{01}$, since the series that define it has the same terms, up to sign, as the series for $\theta$. Note that, unlike the analysis of~\cite{Dupont} for naive theta-constant evaluation, we cannot get a bound for the relative precision: since $\theta(\frac{1+\tau}{2}, \tau) = 0$, there is no lower bound for $|\theta(z, \tau)|$.\footnote{Incidentally, this is why we consider only absolute precision in this paper.} If we set
\begin{equation*}
B(P, \tau) = \left\lceil \sqrt{\frac{P+2}{\pi \Im(\tau) \log_2(e)}} \right\rceil+1.
\end{equation*}
we have $4 |q|^{(B-1)^2} \leq 2^{-P}$, which means the approximation is accurate with absolute precision $P$. We just showed that:
\begin{theorem}
\label{naivecomplex}
To compute $\theta(z, \tau)$ with absolute precision $P$ bits, it is enough to sum over all $k \in \mathbb{Z}$ such that 
\begin{equation*}
|k| \leq \left\lceil \sqrt{\frac{P+2}{\pi \Im(\tau) \log_2(e)}} \right\rceil + 1
\end{equation*}
\end{theorem}
Note that this bound is larger than the one of~\cite[p.~5]{Dupont}.

\subsubsection{Naive algorithm}

We then present a naive algorithm to compute not only the value of $\theta(z, \tau)$, but also the value of $\theta_{01}(z, \tau), \theta_{00}(0, \tau), \theta_{01}(0, \tau)$ for only a marginal amount of extra computation; this is the algorithm we will use for comparison to the fast algorithm we propose in this article. The algorithm performs computations at a precision $\mathcal{P}$, which we determine later so that the result is accurate to the desired precision $P$.

Define the sequence $(v_n)_{n \in \mathbb{N}}$ as
\begin{equation*}
v_n = q^{n^2} (e^{2 i \pi n z} + e^{-2 i \pi n z})
\end{equation*}
so that $\theta(z, \tau) = 1 + \sum_{n \geq 1} v_n$. This satisfies the following recurrence relation for $n>1$:
\begin{equation*}
v_{n+1} = q^{2n} v_1 v_n - q^{4n} v_{n-1}.
\end{equation*}
We use this recursion formula to compute $v_n$ efficiently, which is similar to the trick used by~\cite[Prop.~3]{ThomeEnge}. This removes the need for divisions and the need to compute and store $e^{-2i \pi n z}$, which can get quite big; indeed, computing it only to multiply it by the very small $q^{n^2}$ is wasteful. The resulting algorithm is Algorithm~\ref{naive}.
\begin{algorithm}
    \caption{\label{naive}Compute $\theta_{00, 01}(z, \tau), \theta_{00,01}(0,\tau)$ for $z, \tau$ satisfying conditions~\eqref{eq:conditions}.}
\begin{algorithmic}[1]
  \State $\text{prec} \gets \mathcal{P}$
  \State $B\gets \left\lceil \sqrt{\frac{P+2}{\pi \Im(\tau) \log_2(e)}} \right\rceil + 1$
  \State $\theta_{0,z}\gets 1, \theta_{1,z}\gets 1, \theta_{0,0}\gets 1, \theta_{1,0}\gets 1$
  \State $q\gets e^{i \pi \tau}, \qquad q_1\gets q, \qquad q_2\gets q$
  \State $v_1\gets e^{2 i \pi (z+\tau/2)} + e^{-2 i \pi (z-\tau/2)}, \qquad v\gets v_1, \qquad v'\gets 2$
  \For{$n = 1 .. B$}
    \State /* $q_1 = q^n, q_2 = q^{n^2}, v = v_n, v' = v_{n-1}$ */
    \State $\theta_{0,z}\gets \theta_{0,z} + v, \qquad \theta_{1,z}\gets \theta_{1,z} + (-1)^i \times v$
    \State $\theta_{0,0}\gets \theta_{0,0} + 2 q_2, \qquad \theta_{1,0}\gets \theta_{1,0} + (-1)^i \times 2 q_2$
    \State $q_2\gets q_2 \times (q_1)^2 \times q$
    \State $q_1\gets q_1 \times q$
    \State $\text{temp}\gets v, \qquad v\gets \left( q_1^2 \times v_1 \right) \times v - q_1^4 \times v' \qquad v'\gets \text{temp}$
  \EndFor
\end{algorithmic}
\end{algorithm}

\subsubsection{Error analysis and complexity}

We have the following theorem:
\begin{theorem} \label{naiveanalysis}
For $z, \tau$ satisfying conditions~\eqref{eq:conditions}, Algorithm~\ref{naive} with $\mathcal{P} = P + \log B + 7$ computes $\theta_{00}(z, \tau)$, $\theta_{01}(z, \tau)$, $\theta_{00}(0, \tau)$, $\theta_{01}(0, \tau)$ with absolute precision $P$ bits. This gives an algorithm which has bit complexity $O\left(\mathcal{M}(P) \sqrt{\frac{P}{\Im(\tau)}}\right)$.
\end{theorem}
Performing the analysis of this algorithm requires bounding the error that is incurred during the computation. We then compensate the number of inaccurate bits by increasing the precision. We use the following theorem:
\begin{theorem} \label{theoremlossofprec}
For $j = 1,2$, let $z_j = x_j + i y_j \in \mathbb{C}$ and $\tilde{z_j} = \tilde{x_j} + i \tilde{y_j}$ its approximation. Suppose that $|z_j - \tilde{z_j}| \leq k_j 2^{-P}$ and that $k_j \leq 2^{P/2}$. Then
\begin{enumerate}
\item $|\Re(z_1+z_2) - \Re(\tilde{z_1}+\tilde{z_2})| \leq  (k_1 + k_2) 2^{-P}$
\item $|\Re(z_1z_2) - \Re(\tilde{z_1}\tilde{z_2})| \leq  (2 + 2 k_1|z_2| + 2 k_2|z_1|) 2^{-P}$
\item $|\Re(z_1^2) - \Re(\tilde{z_1}^2)|  \leq  (2 + 4 k_1|z_1|) 2^{-P}$
\end{enumerate}
and the same bounds apply to imaginary parts as well; and
\begin{enumerate}
\setcounter{enumi}{3}
\item $|e^{z_1} - e^{\tilde{z_1}}|  \leq  |e^{z_1}|\frac{7 k_1 + 8.5}{2} 2^{-P} $.
\end{enumerate}
Furthermore if $|z_j| \geq 2 k_j 2^{-P}$,
\begin{enumerate}
\setcounter{enumi}{4}
\item $| \Re\left(\frac{z_1}{z_2}\right) - \Re\left(\frac{\tilde{z_1}}{\tilde{z_2}}\right)|  \leq  \left( \frac{6(2 + 2k_1 |z_2| + 2k_2 |z_1|)}{|z_2|^2} + \frac{2(4 + 8 k_2 |z_2|)(2 |z_1||z_2| + 1) + 2}{|z_2|^4} \right) 2^{-P}$
\end{enumerate}
and the same bound applies to the imaginary part, and
\begin{enumerate}
\setcounter{enumi}{5}
\item $|\sqrt{z_1} - \sqrt{\tilde{z_1}}|  \leq  \frac{k_1}{\sqrt{|z_1|}} 2^{-P} $.
\end{enumerate}
\end{theorem}
This theorem is not very hard to prove; we refer to \cite{absolutelossofprec} for details.
\begin{proof}[Proof of Theorem~\ref{naiveanalysis}]
We first determine the size of the quantities we are manipulating; this is needed to evaluate the error incurred during the computation, as well as the number of bits needed to store fixed-point approximations of absolute precision $P$ of the intermediate quantities. Taking $B=1$ in Proposition~\ref{naiveboundtheta} gives $|\theta(z,\tau)-1| \leq 3$, so $|\theta(z,\tau)| \leq 4$; actually, this also proves $|S_B(z,\tau)|\leq 4$, which means that $|\theta_{0,z}|, |\theta_{1,z}|, |\theta_{0,0}|, |\theta_{1,0}|$ are bounded by 4. We also have $|q| \leq 0.07$, and $|q_2| \leq |q|^{n^2} \leq |q|^n = |q_1| \leq |q| \leq 0.07$. As for the $v_i$, we have $v_0 = 2$, and for $n \geq 1$
\begin{equation*}
|v_n| \leq |q|^{n^2+n} + |q|^{n^2-n} \leq (1+|q|^{2n}) q^{n^2-n} \leq 1.0049 q^{n^2-n} \leq 1.0049
\end{equation*}
Hence, storing all the complex numbers above, including our result, with absolute precision $P$ only requires $P+2$ bits, since their integral part is coded on only 2 bits. Note that, had we computed $e^{-2i \pi n z}$ before multiplying it by $q^{n^2}$, we would have needed $O(\Im(\tau))$ more bits, which worsens the asymptotic complexity.

Computing the absolute precision lost during this computation is done using Theorem~\ref{theoremlossofprec}. We start with the bounds $|\tau - \tilde{\tau}| \leq \frac{1}{2} 2^{-\mathcal{P}}$ and $|z - \tilde{z}| \leq \frac{1}{2} 2^{-\mathcal{P}}$, coming from the hypothesis that the approximations of $z$ and $\tau$ are correctly rounded with precision $\mathcal{P}$. We then need to estimate $k_{v_1}$ and $k_q$, which can be done using the formula giving the absolute error when computing an exponential from Theorem~\ref{theoremlossofprec}. Given that $\tau \in \mathcal{F}$, we have
\begin{eqnarray*}
| q - \tilde{q} | & \leq & 0.07 \frac{7 \times 1/2 + 8.5}{2} 2^{-P} \leq 0.42 \times 2^{-P} \\
| v_1 - \tilde{v_1}| & \leq & 6 ( |e^{-\pi(\Im(\tau)+2\Im(z))}| + |e^{\pi(2\Im(z)-\Im(\tau))}| ) \times 2^{-P} \\
& \leq & 6 (|q|+1) 2^{-P} \leq 6.42 \times 2^{-P}
\end{eqnarray*}
which means that $k_q \leq 0.42$ and $k_{v_1} \leq 6.42$. We then need to evaluate the loss of precision for each variable and at each step of the algorithm, which gives recurrence relations with non-constant coefficients. Solving those is rather tedious, and we use loose upper bounds to simplify the computation; we do not detail this proof in the present article. The results obtained by this method show that the error on the computation of the theta-constants is bounded by $(0.3B + 105.958)2^{-\mathcal{P}}$, and the one on the computation of the theta function is smaller than $(5.894 B + 28.062)2^{-\mathcal{P}}$. This proves that the number of bits lost is bounded by $\log_2 B + c$, where $c$ is a constant smaller than 7; hence we set $\mathcal{P} = P + \log B + 7$.

Finally, evaluating $\pi$ and $\exp(z)$ with precision $\mathcal{P}$ can be done in $O(\mathcal{M}(\mathcal{P}) \log \mathcal{P})$~\cite{Borwein}, but this is negligible asymptotically. In the end, computing an approximation up to $2^{-P}$ of $\theta(z, \tau)$ can be done in $O\left(\mathcal{M}\left(P + \log(P/\Im(\tau)) + c \right) \sqrt{\frac{P}{\Im(\tau)}} \right) = O\left(\mathcal{M}(P) \sqrt{\frac{P}{\Im(\tau)}}\right)$ bit operations.
\end{proof}

\subsubsection{\texorpdfstring{Computing $\theta_{10}$}
		{Computing theta10}
}
\label{section:theta10naive}

We mentioned in section~\ref{section:argumentreduction} the need to compute $\theta_{10}(z,\tau)$ and $\theta_{10}(0, \tau)$ as well. One could think of recovering those values using Jacobi's quartic formula and the equation of the variety:
\begin{eqnarray}
\theta_{00}(0,\tau)^4 &=& \theta_{01}(0,\tau)^4 + \theta_{10}(0,\tau)^4 \label{eq:Jacobi} \\
\theta_{00}^2(z,\tau) \theta_{00}^2(0,\tau) &=& \theta_{01}^2(z,\tau) \theta_{01}^2(0,\tau) + \theta_{10}^2(z,\tau) \theta_{10}^2(z,\tau) \label{eq:Variety}
\end{eqnarray}
that is to say, compute
\begin{eqnarray*}
\theta_{10}(0, \tau) &=& \left( \theta_{00}(0,\tau)^4 - \theta_{01}(0,\tau)^4 \right)^{1/4} \\
\theta_{10}(z, \tau) &=& \frac{\sqrt{\theta_{00}^2(z,\tau) \theta_{00}^2(0,\tau) - \theta_{01}^2(z,\tau) \theta_{01}^2(0,\tau)} }{\theta_{10}(0,\tau)}.
\end{eqnarray*}
However, this approach induces an asymptotically large loss of absolute precision for both $\theta_{10}(0,\tau)$ and $\theta_{10}(z,\tau)$. According to Theorem~\ref{theoremlossofprec}, both square root extraction and inversion induce a loss of precision proportional to $|z|^{-1}$; since $\theta_{10}(0,\tau) \sim 4 q^{1/2}$, the number of bits lost by applying those formulas is $O(\Im(\tau))$. Note that those formulas would also induce a big loss in relative precision; since $\theta_{00}(0,\tau)$ and $\theta_{01}(0,\tau)$ are very close when $\Im(\tau)$ goes to infinity, the subtraction induces a relative precision loss of $O(\Im(\tau))$ bits (for more details, see \cite[Section~6.3]{Dupont}). Either of those analyses show that, in order to compensate precision loss, the naive algorithm should actually be run with a precision of $O(P + \log B + \Im(\tau))$, which gives a running time that worsens, insteads of getting better, when $\Im(\tau)$ gets big. We do not recommend this approach.

Instead, one should compute partial summations of the series defining $\theta_{10}$, much in the same way as we did for $\theta(z,\tau)$. We outline the analysis in this case, which is very similar to the one for $\theta$: supposing $n \geq 2$, we have $n^2-2n \geq (n-2)^2$, which can be used to prove that $|\theta_{10}(z,\tau) - S_B| \leq 3|q|^{(B-2)^2}$, so that the bound on $B$ is thus just one more than for $\theta$; the recurrence relation is the same; $q^{2n}|v_1|$ is bounded by 2 instead of 1, which in the worst case means $\log B$ more guard bits are needed. In what follows, we will refer to this algorithm as ``the naive algorithm to compute $\theta_{10}(0, \tau), \theta_{10}(z,\tau)$''; its asymptotic complexity is, just like Algorithm~\ref{naive}, $O\left(\mathcal{M}(P) \sqrt{\frac{P}{\Im(\tau)}}\right)$ bit operations, which gets better as $\Im(\tau)$ increases.

We note that similar considerations apply to the problem of computing $\theta_{11}$. One can compute $\theta_{11}(z,\tau)$ using the formula~\cite[p.22]{Mumford1}
\begin{equation} \label{eq:recoverTheta11}
\theta_{11}(z,\tau)^2 = \frac{\theta_{01}(z,\tau)^2 \theta_{10}(0,\tau)^2 - \theta_{10}(z,\tau)^2 \theta_{01}(0,\tau)^2}{\theta_{00}(0,\tau)^2}.
\end{equation}
Using this formula loses only a few bits of precision since $\theta_{00}(0,\tau)$ is bounded; however, one then needs to compute a square root, which potentially loses $O(\Im(\tau))$ bits. Hence, a summation of the series, which directly gives $\theta_{11}$, is preferable.

\subsection{Fast computation of theta-constants}
\label{section:AGMthetaconstants}

Recall the definition of the arithmetico-geometric mean (AGM) for two positive real numbers $a,b$:
\begin{eqnarray} \label{eq:AGM}
a_0 = a, && b_0 = b \notag \\
a_{n+1} = \frac{a_n + b_n}{2} && b_{n+1} = \sqrt{a_n b_n}
\end{eqnarray}
The sequences $(a_n)_{n \in \mathbb{N}}$ and $(b_n)_{n \in \mathbb{N}}$ both converge to the same limit, called the \textbf{arithmetico-geometric mean} of $a$ and $b$. Furthermore, $(a_n)$ and $(b_n)$ are quadratically convergent, in the sense of the following definition:
\begin{definition}
A sequence $(a_n)$ is said to be \textbf{quadratically convergent} (to a limit $\ell$) if there is a $C > 0$ such that for $n$ large enough:
\begin{equation*}
|a_{n+1} - a_n | \leq C |a_n - a_{n-1}|^2
\end{equation*}
\end{definition}
The constant $C$ in the case of the AGM can be taken as $\frac{\pi}{8 \min(|a|,|b|)}$~\cite[Thm.~1]{Dupont}. Quadratic convergence implies that the number of exact digits approximately doubles with each iteration, so that one only needs $O(\log P)$ iterations to compute $\AGM(a,b)$ with precision $P$; hence the total cost of computing $\AGM(a,b)$ is $O(\mathcal{M}(P) \log P)$ bit operations.

It is possible to generalize the AGM to complex numbers, but there are two possibilities for the choice of the square root at each step. We then call \textbf{an AGM sequence for $\bm{a}$ and $\bm{b}$} any sequence $(a_n, b_n)_{n \in \mathbb{N}}$ such that
\begin{equation*}
a_0=a, \quad b_0=b, \qquad 2a_{n+1} = a_n+b_n, \quad b_{n+1}^2 = a_n b_n
\end{equation*}
Note that there are uncountably many AGM sequences for $a, b$. We define unambiguously \textit{the} AGM of two complex numbers following~\cite{Cox}:
\begin{proposition}
Let $a,b \in \mathbb{C}$ and let $(a_n, b_n)_{n \in \mathbb{N}}$ be an AGM sequence for $a$ and $b$. We say that the choice of signs is \textbf{good} at the rank $n$ if
\begin{equation*}
| a_n - b_n | < |a_n + b_n| \qquad or \quad |a_n-b_n| = |a_n+b_n| \text{ and } \Im\left(\frac{b_n}{a_n}\right) > 0
\end{equation*}
We call the AGM sequence for $a$ and $b$ in which all the choices of signs are good the \textbf{optimal AGM sequence}, and define $\AGM(a,b)$ as the limit of the optimal AGM sequence for $a$ and $b$.
\end{proposition}
Finally we have the following proposition:
\begin{proposition}[{\protect \cite[Proposition~2.1]{Cox}}]
Let $(a_n, b_n)_{n \in \mathbb{N}}$ be an AGM sequence for $a,b$:
\begin{itemize}
\item If $(a_n, b_n)$ has infinitely many bad choices of sign, $\lim_{n \rightarrow \infty} a_n = 0$ and the convergence is at least linear;
\item If $(a_n, b_n)$ has only finitely many bad choices of sign (for instance if it is optimal), $(a_n)$ and $(b_n)$ both converge quadratically to the same non-zero limit.
\end{itemize}
\end{proposition}

The link between the complex AGM and theta-constants is well-known:
\begin{proposition}
\label{propTauduplication0}
We have the following formulas linking theta-constants:
\begin{eqnarray}
\theta_{00}^2(0, 2\tau) &=& \frac{\theta_{00}^2(0, \tau) + \theta_{01}^2(0, \tau)}{2} \label{eq:TauduplicationConstant00} \\
\theta_{01}^2(0, 2\tau) &=& \theta_{00}(0, \tau) \theta_{01}(0, \tau) \label{eq:TauduplicationConstant01}
\end{eqnarray}
\end{proposition}
This shows that $(\theta_{00}^2(0, 2^n \tau), \theta_{01}(0, 2^n \tau))_{n \in \mathbb{N}}$ is an AGM sequence for $\theta_{00}^2(0, \tau)$ and $\theta_{01}^2(0, \tau)$, and it converges quadratically to 1. Whether or not this sequence is the optimal AGM sequence is controlled by the following result:
\begin{proposition}[{\protect\cite[Theorem~2]{Dupont}} or {\protect\cite[Lemma 2.9]{Cox}}]
Define
\begin{equation*}
\mathcal{F}_{k'} = \left\{ \tau \in \mathcal{H} \text{ such that } |\Re(\tau)|<1, \left|\tau + \frac{3}{4}\right| \geq \frac{1}{4}, \left|\tau + \frac{1}{4}\right| > \frac{1}{4}, \left|\tau - \frac{1}{4}\right| \geq \frac{1}{4}, \left|\tau - \frac{3}{4}\right| > \frac{1}{4} \right\} \subset \mathcal{H}
\end{equation*}
Let $\tau \in \mathcal{F}_{k'}$, and let $(a_n, b_n)_{n \in \mathbb{N}}$ the optimal AGM sequence for $\theta_{00}^2(0, \tau)$ and $\theta_{01}^2(0, \tau)$. Then for all $n$ we have $(a_n, b_n) = (\theta_{00}^2(0, 2^n \tau), \theta_{01}^2(0, 2^n \tau))$, which implies $\AGM(\theta_{00}^2(0, \tau), \theta_{01}^2(0, \tau)) = 1$.
\end{proposition}

In~\cite[Algorithm~4]{Dupont}, an algorithm relying on the AGM, and with complexity $O(\mathcal{M}(P) \log P)$, is given to compute the value of the theta-constants with precision $P$ bits. The algorithm uses the fact that $k'(\tau) = \frac{\theta_{01}^2(0, \tau)}{\theta_{00}^2(0, \tau)}$ is a solution to the following equation:
\begin{equation*} \label{eq:newtonkprime}
i \AGM(1, z) - \tau \AGM\left(1, \sqrt{1 - z^2}\right) = 0
\end{equation*}
which is a consequence of the action of $SL_2(\mathbb{Z})$ on the theta-constants, as well as of Jacobi's quartic formula (Equation~\eqref{eq:Jacobi}). Newton's method, when given an approximation of $k'(\tau)$ with precision $P/2$ as well as the knowledge of $\tau$ with precision $P$, computes an approximation of $k'(\tau)$ with precision $P-\delta$, where $\delta$ is a small constant. If one carries out Newton's method while doubling the working precision with each iteration, it is asymptotically only as costly as the last iteration; this means $k'(\tau)$ can be computed with precision $P$ in quasi-optimal running time. One can then recover the individual values of the theta-constants using the equation
\begin{equation} \label{eq:AGMQuotients}
\AGM\left(1, \frac{\theta_{01}(0, \tau)^2}{\theta_{00}(0, \tau)^2} \right) = \frac{1}{\theta_{00}(0, \tau)^2}.
\end{equation}

However, the complexity of this algorithm is not uniform -- that is to say, it reaches this complexity only for $\tau$ within a compact set. A variant of the algorithm is proposed in~\cite[Algorithm~5]{Dupont} which makes the complexity uniform: if $P \leq 2 \log \Im(\tau)$, use the naive algorithm (which gives the right complexity); if not, compute the value of the theta-constants at $\frac{\tau}{2^n}$ for some $n \leq \log \Im(\tau)$, and use the AGM to compute the theta-constants at $\tau$. This gives an algorithm which complexity does not depend on $\tau$.

\section{A sequence related to \texorpdfstring{$\theta$-functions}{theta-functions}}
\label{section:aboutF}

\subsection{Definition of the F sequence}

We start with the following formula:
\begin{proposition}
\label{propTauduplicationZ}
\begin{eqnarray}
\theta_{00}(z, 2\tau)^2 &=& \frac{\theta_{00}(z,\tau) \theta_{00}(0,\tau) + \theta_{01}(z, \tau) \theta_{01}(0, \tau)}{2} \label{eq:Tauduplication00}  \\
\theta_{01}(z, 2\tau)^2 &=& \frac{\theta_{00}(z,\tau) \theta_{01}(0,\tau) + \theta_{01}(z, \tau) \theta_{00}(0, \tau)}{2} \label{eq:Tauduplication01}
\end{eqnarray}
\end{proposition}

This formula is called in \cite[formula 3.13, p.39]{CossetPhd} the change of basis formula from the $\mathcal{F}_2$ basis to the $\mathcal{F}_{(n,2)^2}$ basis. However a direct proof can be obtained with limited effort, using the series defining $\theta$ and some manipulations and term reorganization akin to
\begin{equation*}
\sum_{n + m \equiv 0 \pmod{2}} q^{n^2 + m^2} = \sum_{i, j \in \mathbb{Z}} q^{(i+j)^2 + (i-j)^2} .
\end{equation*}
We also note that one can similarly prove the following formula, which will be used in section~\ref{section:uniformalgorithm}:
\begin{equation} \label{eq:Tauduplication10}
\theta_{10}(z, 2\tau)^2 = \frac{\theta_{00}(z,\tau) \theta_{00}(0,\tau) - \theta_{01}(z,\tau) \theta_{01}(0,\tau)}{2}
\end{equation}

We then define the following function:
\begin{eqnarray*}
F : \mathbb{C}^4 & \rightarrow & \mathbb{C}^4 \\
(x,y,z,t) & \mapsto & \left( \frac{\sqrt{x} \sqrt{z} + \sqrt{y} \sqrt{t}}{2}, \frac{\sqrt{x} \sqrt{t} + \sqrt{y} \sqrt{z}}{2}, \frac{z+t}{2},  \sqrt{z}\sqrt{t} \right)
\end{eqnarray*}
Hence, according to Proposition~\ref{propTauduplication0} and~\ref{propTauduplicationZ}, for some appropriate choice of roots we have
\begin{equation*}
F\left( \theta_{00}^2(z,\tau), \theta_{01}^2(z,\tau), \theta_{00}^2(0,\tau), \theta_{01}^2(0,\tau) \right) = \left( \theta_{00}^2(z,2\tau), \theta_{01}^2(z,2\tau), \theta_{00}^2(0,2\tau), \theta_{01}^2(0,2\tau) \right).
\end{equation*}

\begin{proof}[Remark]
One can also write rewrite $F$ using Karatsuba-like techniques
\begin{eqnarray} \label{FKaratsuba}
F(x,y,z,t) = && \left( \frac{(\sqrt{x}+\sqrt{y})(\sqrt{z}+\sqrt{t}) + (\sqrt{x}-\sqrt{y})(\sqrt{z}-\sqrt{t})}{4}, \right. \\
&& \left. \frac{(\sqrt{x}+\sqrt{y})(\sqrt{z}+\sqrt{t}) - (\sqrt{x}-\sqrt{y})(\sqrt{z}-\sqrt{t})}{4}, \right. \notag \\
&& \left. \frac{(\sqrt{z}+\sqrt{t})^2 + (\sqrt{z}-\sqrt{t})^2}{4}, \frac{(\sqrt{z}+\sqrt{t})^2 - (\sqrt{z}-\sqrt{t})^2}{4} \right) \notag
\end{eqnarray}
to speed up computations.
\end{proof}

Following section~\ref{section:AGMthetaconstants}, we define a \textbf{good choice} for square roots at the rank $n$ as the following conditions being satisfied:
\begin{itemize}
\item $\Re(\sqrt{x_n}) \geq 0$, \qquad $\Re(\sqrt{z_n}) \geq 0$;
\item $|\sqrt{x_n} - \sqrt{y_n}| < |\sqrt{x_n} + \sqrt{y_n}|$, \quad or $|\sqrt{x_n} - \sqrt{y_n}| = |\sqrt{x_n} + \sqrt{y_n}|$ and $\Im\left(\frac{\sqrt{y_n}}{\sqrt{x_n}}\right) > 0$;
\item $|\sqrt{z_n} - \sqrt{t_n}| < |\sqrt{z_n} + \sqrt{t_n}|$, \qquad or $|\sqrt{z_n} - \sqrt{t_n}| = |\sqrt{z_n} + \sqrt{t_n}|$ and $\Im\left(\frac{\sqrt{t_n}}{\sqrt{z_n}}\right) > 0$.
\end{itemize}
The last condition is equivalent to $|z_n-t_n| \leq |z_n+t_n|$, which means that $(z_n, t_n)$ is an AGM sequence for $(z_0, t_0)$ in which all the choices of sign are good. Note that the condition $|x-y| < |x+y|$ is equivalent to $\Re\left(\frac{y}{x}\right) > 0$.

Again, similarly to the AGM, for any $x,y,z,t \in \mathbb{C}$ we define the \textbf{optimal F sequence} $((x_n, y_n, z_n, t_n))_{n \in \mathbb{N}}$ as follows:
\begin{eqnarray*}
(x_0, y_0, z_0, t_0) &=& (x,y,z,t) \\
(x_{n+1}, y_{n+1}, z_{n+1}, t_{n+1}) &=& F\left(x_n, y_n, z_n, t_n\right)
\end{eqnarray*}
where all the choices of sign for the square roots are good. The study of this sequence and its convergence is done in Section~\ref{section:convergence}.

\subsection{Link with \texorpdfstring{$\theta$-functions}{theta-functions}}

\subsubsection{More argument reduction}
\label{section:moreArgReduction}

We go slightly further than the conditions~\eqref{eq:conditions} in order to justify the forthcoming results. We wish to further reduce $z$, as follows:
\begin{equation} \label{eq:conditionsz}
0 \leq \Im(z) \leq  \frac{\Im(\tau)}{4}, \qquad |\Re(z)| \leq \frac{1}{8}
\end{equation}
The first hypothesis allows us to avoid $z = \frac{\tau+1}{2}$, which is a zero of $\theta(z,\tau)$, and hence a pole of quotients of the form  $\frac{\theta_i}{\theta_{00}}$, which we consider in our algorithm much in the same way as~\cite{Dupont}. We prove Lemma~\ref{realPartsOfQuotients} and Theorem~\ref{goodchoicethetas} under this assumption. The second condition complements the first one as follows:
\begin{lemma}
Let $z, \tau$ such that $|\Re(\tau)| \leq \frac{1}{2}$ and the conditions of~\eqref{eq:conditionsz} are verified. Then $\frac{z}{\tau}, \frac{-1}{\tau}$ verify the first condition of~\eqref{eq:conditionsz}.
\end{lemma}
\begin{proof}
Write
\begin{equation*} 
\left|\Im\left( \frac{z}{\tau} \right) \right| = \frac{1}{|\tau|^2} |\Im(z) \Re(\tau) - \Re(z) \Im(\tau)| \leq \frac{\Im(\tau)}{|\tau|^2} \left(\frac{1}{4} \frac{1}{2} + \frac{1}{8}\right) = \frac{1}{4} \Im\left(\frac{-1}{\tau}\right).
\end{equation*}
\end{proof}
This will be used to apply Theorem~\ref{goodchoicethetas} to $\frac{z}{\tau}, \frac{-1}{\tau}$ in Proposition~\ref{propmathfrakF}.

Note that those conditions are satisfied if one takes $z' = \frac{z}{2}$ with $z$ satisfying conditions~\eqref{eq:conditions}. One can then compute $\theta_{00,01}^2(z,\tau)$ from $\theta_{00,01,10}^2(z/2,\tau)$ and $\theta_{00,01,10}^2(0,\tau)$ using the following \textit{$z$-duplication formulas}~\cite[p.~22]{Mumford1}:
\begin{eqnarray} \label{eq:Zduplication}
\theta_{00}(z, \tau) \theta_{00}^3(0,\tau) &=& \theta_{01}^4(z/2, \tau) + \theta_{10}^4(z/2, \tau) \\
\theta_{01}(z, \tau) \theta_{01}^3(0,\tau) &=& \theta_{00}^4(z/2, \tau) - \theta_{10}^4(z/2, \tau)  \notag \\
\theta_{10}(z, \tau) \theta_{10}^3(0,\tau) &=& \theta_{00}^4(z/2, \tau) - \theta_{01}^4(z/2, \tau). \notag
\end{eqnarray}
This requires the knowledge of $\theta_{10}(z,\tau)$ and the associated theta-constant; this could be computed using Jacobi's formula (Equation~\eqref{eq:Jacobi}) and the equation of the variety (Equation~\eqref{eq:Variety}), but we end up using a different trick in our final algorithm.

\subsubsection{Good choices of sign and thetas}

We now prove that, for the arguments we consider, the good choices of sign correspond exactly to values of $\theta$:
\begin{lemma} \label{realPartsOfQuotients}
For any $\tau$ such that $\Im(\tau) \geq \Value$ (in particular, for $\tau \in \mathcal{F}$) and $z$ verifying the first condition in~\eqref{eq:conditionsz} we have
\begin{equation*}
|\theta_{00}(z,\tau) - \theta_{01}(z,\tau)| < |\theta_{00}(z,\tau) + \theta_{01}(z,\tau)|,
\end{equation*}
which also proves that $\Re\left(\frac{\theta_{01}(z,\tau)}{\theta_{00}(z,\tau)}\right) > 0, \Re\left(\frac{\theta_{01}(0,\tau)}{\theta_{00}(0,\tau)}\right) > 0$.
\end{lemma}
\begin{proof}
Write:
\begin{eqnarray*}
|\theta_{00}(z,\tau) + \theta_{01}(z,\tau) - 2| & \leq & 2 \sum_{n \geq 2, n \text{ even}} |q^{n^2} (w^{2n}+w^{-2n})| \\
& \leq & 2 \sum_{n \geq 2, n \text{ even}} |q|^{n^2} (1 + |q|^{-n/2}) \\
& \leq & 2 \sum_{n \geq 1} |q|^{4n^2} (1 + |q|^{-n}) \\
& \leq & 2 |q|^3 + 2 |q|^4 + \frac{2q^{16}}{1-q^{20}} + \frac{2q^{14}}{1-q^{19}} \\
|\theta_{00}(z,\tau) - \theta_{01}(z,\tau)| & \leq & 2 \sum_{n \geq 1, n \text{ odd}} |q|^{n^2} (1 + |q|^{-n/2}) \\
& \leq & 2 |q|^{1/2} + 2|q| + \frac{2q^9}{1-q^{16}} + \frac{2q^{7.5}}{1-q^{19}}
\end{eqnarray*}
We have
\begin{equation*}
\frac{2 |q|^{1/2} + 2|q| + \frac{2q^9}{1-q^{16}} + \frac{2q^{7.5}}{1-q^{19}}}{2 - (2 |q|^3 + 2 |q|^4 + \frac{2q^{16}}{1-q^{20}} + \frac{2q^{14}}{1-q^{19}})} \leq 1
\end{equation*}
for $\Im(\tau) > \Value$, which proves the lemma.
\end{proof}
We are now ready to prove:
\begin{theorem}
\label{goodchoicethetas}
Let $(x_n, y_n, z_n, t_n)$ be the optimal F sequence for $\theta_{00}^2(z,\tau), \theta_{01}^2(z,\tau), \theta_{00}^2(0,\tau), \theta_{01}^2(0,\tau)$.  For any $\tau$ such that $\Im(\tau) \geq \Value$ and $z$ satisfying the first condition of~\eqref{eq:conditionsz} we have
\begin{equation*}
(x_n, y_n, z_n, t_n) = \left( \theta_{00}^2(z, 2^n\tau), \theta_{01}^2(z, 2^n \tau), \theta_{00}^2(0, 2^n \tau), \theta_{01}^2(0, 2^n \tau) \right)
\end{equation*}
\end{theorem}
\begin{proof}
This is true for $n=0$; we prove the statement inductively. Suppose it is true for $n=k$. We have for any $\tau$:
\begin{equation*}
\theta_{00}(0, \tau) = 1 + 2q + c, \qquad |c| \leq \frac{2|q|^4}{1-|q|^5}
\end{equation*}
For any $\tau$ such that $\Im(\tau) \geq \Value$, $2|q| \leq 0.676$ and $|c| \leq 0.027$; hence $\Re(\theta_{00}(0,2^k \tau)) > 0$ for any $k$, which proves that $\sqrt{z_k}=\theta_{00}(0,2^k\tau)$. Lemma~\ref{realPartsOfQuotients} shows that $\Re\left( \frac{\theta_{01}(0,\tau)}{\theta_{00}(0,\tau)} \right) > 0$, and we also have $\Re\left( \frac{\sqrt{t_k}}{\sqrt{z_k}} \right) \geq 0$ since the choice of roots is good, hence $\sqrt{t_k}=\theta_{01}(0, 2^k\tau)$. Proposition~\ref{propTauduplication0} then proves that $t_{k+1}=\theta_{01}^2(0,2^{k+1}\tau)$ and $z_{k+1}=\theta_{00}^2(0, 2^{k+1} \tau)$.

Similarly, given that $z$ satisfies the first condition of~\eqref{eq:conditionsz}:
\begin{equation} \label{eq:estimateTheta}
|\theta_{00}(z, \tau) - 1| \leq |q|^{1/2} + |q| + |q|^3 + |q|^4 + |q|^{7/2} + |q|^9 + \frac{2|q|^{14}}{1-|q|^2}
\end{equation}
For $\Im(\tau) \geq \Value$, this is strictly smaller than 1; hence $\Re(\theta_{00}(z,\tau)) > 0$, which proves that $\sqrt{x_k} = \theta_{00}(z, 2^k \tau)$. Again, Lemma~\ref{realPartsOfQuotients} proves that $\Re\left( \frac{\theta_{01}(z,2^k \tau)}{\theta_{00}(z, 2^k\tau)} \right) > 0$, and since the choice of signs is good, $\Re\left( \frac{\sqrt{y_k}}{\sqrt{x_k}} \right) \geq 0$, necessarily $\sqrt{y_k} = \theta_{01}(z, 2^k \tau)$. This along with Proposition~\ref{propTauduplicationZ} finishes the induction.
\end{proof}

Note that a consequence of this proposition is the following fact:
\begin{proposition} \label{thetaQuadraticallyConvergent}
The optimal F sequence for $\theta_{00}^2(z,\tau), \theta_{01}^2(z,\tau), \theta_{00}^2(0,\tau), \theta_{01}^2(0,\tau)$ converges quadratically to $(1,1,1,1)$.
\end{proposition}

\subsection{A function with quasi-optimal time evaluation}

The strategy of~\cite{Dupont} is to use an homogenization of the AGM to get a function $f_{\tau}:~\mathbb{C}~\rightarrow~\mathbb{C}$, on which Newton's method can be applied. To generalize this, we homogenize the function which maps to $(x,y,z,t)$ the limit of the optimal F sequence associated to them; it becomes a function from $\mathbb{C}^2$ to $\mathbb{C}^2$. We call this function $F^{\infty}$; this function is a major building block for the function we use to compute our two parameters $z, \tau$ using Newton's method.

\begin{proposition} \label{calcullambda}
Let $\lambda, \mu \in \mathbb{C}$. Let $((x_n, y_n, z_n, t_n))_{n \in \mathbb{N}}$ be the optimal F sequence for $(x,y,z,t)$, and $((x'_n, y'_n, z'_n, t'_n))_{n \in \mathbb{N}}$ the optimal F sequence for $(\lambda x,\lambda y,\mu z,\mu t)$. Put $\lim_{n \rightarrow \infty} z_n = z_{\infty}$ and $\lim_{n \rightarrow \infty} z'_n = z'_{\infty}$. Then we have
\begin{equation*}
\mu = \frac{z'_{\infty}}{z_{\infty}}, \qquad \qquad \lambda = \frac{ \left( \lim_{n \rightarrow \infty} \left( \frac{x'_n}{z'_{\infty}} \right)^{2^n} \right) \times z'_{\infty}}{\left( \lim_{n \rightarrow \infty} \left( \frac{x_n}{z_{\infty}} \right)^{2^n} \right) \times z_{\infty}}
\end{equation*}
\end{proposition}
\begin{proof}
We prove by induction that
\begin{equation*}
x'_n = \epsilon_n \lambda^{1/2^n} \mu^{1-1/2^n} x_n, \quad y'_n = \epsilon_n \lambda^{1/2^n} \mu^{1-1/2^n} y_n, \qquad z'_n = \mu z_n, \quad t'_n = \mu t_n,
\end{equation*}
where $\Re(\lambda^{1/2^n}) \geq 0$, $\Re(\mu^{1-1/2^n}) \geq 0$, and $\epsilon_n$ is a $2^n$-th root of unity. This is enough to prove the proposition above, since then
\begin{equation*}
\lim_{n \rightarrow \infty} \left( \frac{x'_n}{z'_{\infty}} \right)^{2^n} = \lim_{n \rightarrow \infty} \lambda \mu^{2^n-1} \left( \frac{x_n}{z'_{\infty}} \right)^{2^n} = \frac{\lambda}{\mu} \lim_{n \rightarrow \infty} \left( \frac{x_n}{z_{\infty}} \right)^{2^n}.
\end{equation*}

Since this is true for $n=0$, suppose this is true for $n=k$. 
We have
\begin{equation*}
z'_{k+1} = \frac{z'_k + t'_k}{2} = \mu z_{k+1}.
\end{equation*}
As for $t_{k+1}$, we can write
\begin{equation*}
\sqrt{z'_k} = \epsilon_z \sqrt{\mu} \sqrt{z_k}, \qquad \sqrt{t'_k} = \epsilon_t \sqrt{\mu} \sqrt{t_k}
\end{equation*}
where $\epsilon_z = \pm 1$ and $\epsilon_t = \pm 1$, and the square roots are taken with positive real part. But since $\Re\left(\frac{\sqrt{t'_k}}{\sqrt{z'_k}}\right) \geq 0$ and $\Re\left(\frac{\sqrt{t_k}}{\sqrt{z_k}} \right) \geq 0$, we have $\epsilon_z = \epsilon_t$. Hence
\begin{equation*}
t'_{k+1} = \left( \epsilon_z \sqrt{\mu} \sqrt{z_k} \right) \left( \epsilon_z \sqrt{\mu} \sqrt{t_k} \right) = \mu t_{k+1}.
\end{equation*}

As for the other coordinates, we have
\begin{equation*}
\sqrt{x'_k} = \epsilon_x \sqrt{\epsilon_k} \lambda^{1/2^{k+1}} \mu^{1/2-1/2^{k+1}} \sqrt{x_k}, \qquad \sqrt{y'_k} = \epsilon_y \sqrt{\epsilon_k} \lambda^{1/2^{k+1}} \mu^{1/2-1/2^{k+1}} \sqrt{y_k}
\end{equation*}
where the roots are taken with positive real part, and $\epsilon_x, \epsilon_y \in \{-1, 1\}$. Since $\Re\left( \frac{\sqrt{y_k}}{\sqrt{x_k}} \right) \geq 0$ and we require $\Re\left( \frac{\sqrt{y'_k}}{\sqrt{x'_k}} \right) \geq 0$, necessarily $\epsilon_x = \epsilon_y$; hence
\begin{equation*}
x'_{k+1} = \frac{\sqrt{x'_k} \sqrt{z'_k} + \sqrt{y'_k} \sqrt{t'_k}}{2} = \epsilon_{k+1} \lambda^{1/2^{k+1}} \mu^{1-1/2^{k+1}} \frac{\sqrt{x_k} \sqrt{z_k} + \sqrt{y_k} \sqrt{t_k}}{2} = \epsilon_{k+1} \lambda^{1/2^{k+1}} \mu^{1-1/2^{k+1}} x_{k+1}
\end{equation*}
where $\epsilon_{k+1} = \epsilon_x \sqrt{\epsilon_z}$ is indeed such that $\epsilon_{k+1}^{2^{k+1}} = 1$. This proves the proposition.
\end{proof}

In the case of theta-functions, however, we have:
\begin{proposition} \label{thetasPow2PowN}
\begin{equation*}
\lim_{n \rightarrow \infty} \frac{\theta(z,2^n\tau)^{2^n}}{\theta(0,2^n\tau)^{2^n}} = 1
\end{equation*}
\end{proposition}
\begin{proof}
It is enough to prove $\lim_{n \rightarrow \infty} \theta(z,2^n\tau)^{2^n} = 1$, since it also covers the case $z=0$. Write, like Equation~\eqref{eq:estimateTheta}:
\begin{equation*}
\theta(z, 2^n\tau) = 1 + q^{2^n}(w^2+w^{-2}) + c, \qquad |c| \leq \frac{2|q|^{2^n+3}}{1-|q|}
\end{equation*}
We can write $\theta(z,2^n\tau) = 1 + d_n$ with $|d_n| \leq 2 |q|^{2^n}|w^2+w^{-2}|$. We then have classically
\begin{equation*}
|\theta(z,2^n\tau)^{2^n} - 1| \sim |2^n d_n|
\end{equation*}
and since $\lim_{n \rightarrow \infty} 2^n |q|^{2^n} = 0$, this proves the proposition.
\end{proof}
Combining this with Theorems~\ref{calcullambda} and~\ref{goodchoicethetas} proves that, for $(x_n,y_n,z_n,t_n)$ as in Theorem~\ref{goodchoicethetas}, we have
\begin{equation*}
\lambda = \lim_{n \rightarrow \infty} \left( \frac{x'_n}{\mu} \right)^{2^n} \times \mu.
\end{equation*}
In particular, if we define the following function:
\begin{eqnarray*}
F^{\infty} : \mathbb{C}^4 & \rightarrow & \mathbb{C}^2 \\
(x, y, z, t) & \mapsto & \left(\left( \lim_{n \rightarrow \infty} \left( \frac{x_n}{z_{\infty} } \right)^{2^n} \right) \times z_{\infty}, z_{\infty} \right)
\end{eqnarray*}
where $z_{\infty} = \lim_{n \rightarrow \infty} z_n$, then we have that, for any $z, \tau$ satisfying the hypotheses of Theorem~\ref{goodchoicethetas},
\begin{equation*}
F^{\infty}\left(\lambda \theta_{00}^2(z,\tau), \lambda \theta_{01}^2(z,\tau), \mu \theta_{00}^2(0,\tau), \mu \theta_{01}^2(0,\tau) \right) = (\lambda, \mu).
\end{equation*}
For instance,
\begin{equation} \label{eq:OurAGMQuotients}
F^{\infty}\left(1, \frac{\theta_{01}^2(z,\tau)}{\theta_{00}^2(z,\tau)} , 1, \frac{\theta_{01}^2(0,\tau)}{\theta_{00}^2(0,\tau)} \right) = \left( \frac{1}{\theta_{00}^2(z,\tau)}, \frac{1}{\theta_{00}^2(0,\tau)} \right).
\end{equation}
This is similar to Equation~\eqref{eq:AGMQuotients}, and will play a similar role in the computation of $\theta(z,\tau)$.

\subsection{Convergence}
\label{section:convergence}

Let us start by showing that, contrary to the AGM and despite Proposition~\ref{thetaQuadraticallyConvergent}, an optimal F sequence does not always converge quadratically; for instance, the optimal F sequence for $(2,2,1,1)$ is $((2^{1/2^n}, 2^{1/2^n}, 1, 1))_{n \in \mathbb{N}}$, which does not converge quadratically. This is a big difference from the AGM, and this is why we are reluctant to call optimal F sequences a ``generalization of the AGM''. However, we now show that the sequence $(\lambda_n) = \left(\left(\frac{x_n}{z_{\infty}}\right)^{2^n} \times z_{\infty} \right)_{n \in \mathbb{N}}$ converges quadratically, whence $F^{\infty}$ can be computed in $O(\mathcal{M}(P) \log P)$ bit operations.

In all that follows, we take $C > 1$ such that $|x_0|, |y_0|, |z_0|, |t_0| \leq C$. Note that the equations defining $F$ can then be used to prove inductively that
\begin{equation*}
\forall n \in \mathbb{N}, |x_n|, |y_n|, |z_n|, |t_n| \leq C.
\end{equation*}
We prove in Section~\ref{section:proveAllHypo} that there is a suitable $C$ for any $z, \tau$ we consider in our final algorithm. We also have a lower bound:
\begin{proposition}
There exists $\epsilon, n_0$ such that for all $n \geq n_0$, $|x_n|, |y_n|, |z_n|, |t_n| \geq \epsilon$.
\end{proposition}
\begin{proof}
Note that \cite[Prop~2.1]{Cox} proves that $|z_n|, |t_n| \geq c$, with
$c=\frac{2}{\pi} \min(|z_0|, |t_0|)$. Suppose $|x_n|, |y_n|, |z_n|, |t_n| \geq \delta > 0$, and write
$$|\sqrt{x_n}+\sqrt{y_n}|^2 = 2|\sqrt{x_n}|^2 + 2|\sqrt{y_n}|^2 -
|\sqrt{x_n}-\sqrt{y_n}|^2.$$
Using $|\sqrt{x_n}-\sqrt{y_n}| \leq |\sqrt{x_n}+\sqrt{y_n}|$, which comes
the fact that choices of square roots are good, we derive the two
following lower bounds.
\begin{align*}
|\sqrt{x_n}+\sqrt{y_n}| &\geq \sqrt{ 2|\sqrt{x_n}|^2 + 2|\sqrt{y_n}|^2 - |x_n-y_n|}, \\
& \geq 2 \sqrt{\delta} \sqrt{1-\frac{|x_n-y_n|}{4 \delta}},\\
    \llap{and }
|\sqrt{x_n}+\sqrt{y_n}|^2 &\geq { 2|\sqrt{x_n}|^2 + 2|\sqrt{y_n}|^2 -
|\sqrt{x_n}+\sqrt{y_n}|^2},\\
|\sqrt{x_n}+\sqrt{y_n}|^2 &\geq { |\sqrt{x_n}|^2 + |\sqrt{y_n}|^2} \geq
2\delta.
\end{align*}
Using the same arguments, we obtain:
\begin{align*}
|\sqrt{z_n}+\sqrt{t_n}| &\geq 2 \sqrt{c} \sqrt{1 - \frac{|z_n-t_n|}{4 c}},\\
|\sqrt{z_n}+\sqrt{t_n}| &\geq \sqrt{2c}.
\end{align*}
We then have
\begin{align*}
|\sqrt{x_n}-\sqrt{y_n}| & \leq \frac{|x_n-y_n|}{|\sqrt{x_n}+\sqrt{y_n}|} \leq \frac{|x_n-y_n|}{\sqrt{2 \delta} }, \\
|\sqrt{z_n}-\sqrt{t_n}| & \leq \frac{|z_n-t_n|}{|\sqrt{z_n}+\sqrt{t_n}|} \leq \frac{|z_n-t_n|}{\sqrt{2 c} }
\end{align*}
which we use along with Equation~\ref{FKaratsuba} to write:
\begin{equation*}
|x_{n+1}| \geq \sqrt{\delta} \sqrt{c} \sqrt{\left( 1-\frac{|x_n-y_n|}{4 \delta} \right) \left( 1 - \frac{|z_n-t_n|}{4 c} \right)} - \frac{|x_n-y_n||z_n-t_n|}{8 \sqrt{\delta} \sqrt{c} }
\end{equation*}
Now, using only \cite[Prop~2.1]{Cox} and Equation~\ref{FKaratsuba}, we can write
\begin{equation*}
|x_{n+1}-y_{n+1}| = \frac{|\sqrt{x_n}-\sqrt{y_n}| |\sqrt{z_n} -
\sqrt{t_n}|}{2} \leq \sqrt{\frac{C}{2c}} |z_n-t_n|
\end{equation*}
which proves that $(|x_n-y_n|)$ converges quadratically to 0; we thus have $|x_n - y_n| \leq C_n \frac{\delta}{2^{2^n}}$, $|z_n - t_n| \leq C'_n \frac{c}{2^{2^n}}$ with $C_n, C'_n = o(2^{2^n})$. This gives
\begin{equation*}
|x_{n+1}| \geq \sqrt{\delta} \sqrt{c} \left( \sqrt{\left( 1-\frac{C_n}{2^{2^n+2}} \right) \left( 1 - \frac{C'_n}{2^{2^n+2}} \right)} - \frac{C_n}{2^{2^n+3}} \frac{C'_n}{2^{2^n+3}} \right)
\end{equation*}
Finally, supposing that $c > \delta$, we get that $|x_n| \geq \delta \Rightarrow |x_{n+1}| \geq \delta \epsilon_n$; since $(\epsilon_n)$ converges to 1 quadratically, $\prod_{k \geq n} \epsilon_n$ converges to $\epsilon > 0$, which proves that $|x_{n+k}| \geq \epsilon \delta$ for all $k$.
\end{proof}

We now prove that $(\lambda_n) = \left(\left(\frac{x_n}{z_{\infty}}\right)^{2^n} \times z_{\infty} \right)_{n \in \mathbb{N}}$ converges quadratically, by proving the following theorem:
\begin{theorem} \label{thmconvergence}
The sequence $(\lambda_n)$ converges, to a limit $\lambda$. Furthermore, for $P$ large enough, there exists a constant $c_1 > 0$, depending on $C, \epsilon$ and $|\lambda|$, such that, if $k$ is the first integer such that $|z_k - t_k| \leq 2^{-P-k-c_1}$, then $\lambda_{k+1}$ is an approximation of $\lambda$ with absolute precision $P$ bits.
\end{theorem}
\begin{proof}
The point here is that once $z_n$ and $t_n$ are close enough, $x_{n+1}$ and $y_{n+1}$ are also close and the value of $\lambda_n$ does not change much after that. Let $c_1 \geq 0$, and take $n$ the first integer for which $|z_n - t_n| \leq \eta$ with $\eta = 2^{-P - c_1 - n}$. We then have for all $k \geq 0$~\cite[Theorem~1]{Dupont}:
\begin{equation*}
|z_{n+k} - t_{n+k} | \leq A^{2^k-1} \eta^{2^k}
\end{equation*}
with $A = \frac{\pi}{8 \min(|z_0|, |t_0|)}$. Furthermore, $|z_{n+1}-z_n| = \frac{1}{2} |z_n - t_n|$, so that
\begin{equation*}
|z_{\infty} - z_{n+k}| \leq \frac{1}{2} \sum_{i=k}^{\infty} A^{2^i-1} \eta^{2^i}
\end{equation*}
and we have $|z_{\infty} - z_{n+k}| \leq \frac{1}{A} \left(A \eta \right)^{2^k}$. Finally, using Equation~\eqref{FKaratsuba}, one can write
\begin{eqnarray*}
|x_{n+k+1}-y_{n+k+1}| & \leq & \frac{|\sqrt{x_n}-\sqrt{y_n}||\sqrt{z_n}-\sqrt{t_n}|}{2} \\
& \leq & \sqrt{C} \sqrt{2} \sqrt{|z_{n+k+1}-t_{n+1}|} \text{ \qquad since $z_{n+k+1}-t_{n+k+1} = \frac{(\sqrt{z_{n+k}}-\sqrt{t_{n+k}})^2}{2}$ } \\
& \leq & \sqrt{2AC} |z_{n+k} - t_{n+k}|.
\end{eqnarray*}

\medskip
Now, define $q_n = \frac{(x_n/z_{\infty})^2}{x_{n-1}/z_{\infty}}$, so that $\frac{\lambda_{n+1}}{\lambda_n} = q_n^{2^n}$. Note that if one makes the approximation $x_{n+k+1} = y_{n+k+1}$ and $z_{n+k+1} = t_{n+k+1} = z_{\infty}$, we have $x_{n+k+2} = \sqrt{x_{n+k+1} z_{\infty}}$ which gives $q_{n+k+2} = 1$. We take a closer look at those approximations:
\begin{eqnarray*}
|x_{n+k+2} - \sqrt{x_{n+k+1}} \sqrt{z_{n+k+1}}| & \leq & \frac{|\sqrt{y_{n+k+1}}-\sqrt{x_{n+k+1}}| |\sqrt{z_{n+k+1}}+\sqrt{t_{n+k+1}}|}{4} \\
&& + \frac{|\sqrt{y_{n+k+1}}+\sqrt{x_{n+k+1}}| |\sqrt{z_{n+k+1}}-\sqrt{t_{n+k+1}}|}{4} \\
& \leq & \frac{\sqrt{C}}{2} (|\sqrt{y_{n+k+1}}-\sqrt{x_{n+k+1}}| + |\sqrt{z_{n+k+1}}-\sqrt{t_{n+k+1}}|) \\
& \leq & \frac{\sqrt{C}}{2} \left( \frac{\sqrt{2AC}}{|\sqrt{x_{n+k}}+\sqrt{y_{n+k}}|} |z_{n+k}-t_{n+k}| + \sqrt{2A} |z_{n+k+1}-t_{n+k+1} | \right) \\
& \leq & C' (A \eta)^{2^k}
\end{eqnarray*}
so
\begin{eqnarray*}
q_{n+k+2} - 1 &=& \frac{(x_{n+k+2}/z_{\infty})^2 - x_{n+k+1}/z_{\infty}}{x_{n+k+1}/z_{\infty}} \\
& \leq & \frac{\frac{1}{z_{\infty}^2} (\sqrt{x_{n+k+1}} \sqrt{z_{n+k+1}} + C'(A \eta)^{2^k})^2 - x_{n+k+1}/z_{\infty}}{x_{n+k+1}/z_{\infty}} \\
& \leq & \frac{x_{n+k+1} z_{n+k+1}/z_{\infty}^2 - x_{n+k+1} / z_{\infty}}{x_{n+k+1}/z_{\infty}} + \frac{2C/z_{\infty}^2 C'(A \eta)^{2^k} + C'^2 (A \eta)^{2^{k+1}}}{x_{n+k+1}/z_{\infty}} \\
& \leq & \frac{1}{2 z_{\infty}} \sum_{i=k+1}^{\infty} (A \eta)^{2^i} + \frac{C C'(A \eta)^{2^k}}{x_{n+k+1} z_{\infty}} + \frac{C'^2 (A \eta)^{2^{k+1}}}{x_{n+k+1}/z_{\infty}} \\
& \leq & c \times (A \eta)^{2^k}
\end{eqnarray*}
This proves that $(q_n)$ converges quadratically to 1; using the equivalent $q_n^{2^n} - 1 \sim 2^n q_n$, we have that $(q_n^{2^n})$ also converges quadratically to 1, which proves the convergence of the sequence $(\lambda_n)$.

Finally we have
\begin{eqnarray*}
q_{n+2}^{2^{n+2}} ... q_{n+k}^{2^{n+k}} - 1 &\leq & \prod_{i=0}^{k-2} \left(1+ c (A\eta)^{2^i}\right)^{2^{n+i+2}} - 1 \\
& < & \exp\left(4 c \sum_{i=0}^{k-2} 2^{n+i} (A \eta)^{2^i}\right) -1 \\
& \leq & \frac{4 c \sum_{i=0}^{k-2} 2^{n+i} (A\eta)^{2^i}}{1-(4 c \sum_{i=0}^{k-2} 2^{n+i} (A \eta)^{2^i})/2}
\end{eqnarray*}
which proves, for $P$ large enough,
\begin{eqnarray*}
|\lambda - \lambda_{n+1}| & \leq & 8  c |\lambda_{n+1}|  \sum_{i=0}^{\infty} 2^{n+i} (A\eta)^{2^i} \\
& \leq & 16 c |\lambda_{n+1}| A \eta 2^n \\
& \leq & 16 A c |\lambda_{n+1}| 2^{-c_1} \times 2^{-P}
\end{eqnarray*}
This inequality proves that, at least for $P$ large enough, $|\lambda_{n+1}| \leq 2 |\lambda|$. Hence if we suppose $\log_2(32 A c |\lambda|) \leq c_1$, we have that $\lambda_{n+1}$ is an approximation of $\lambda$ with $P$ bits of absolute precision.
\end{proof}

\begin{algorithm}
\caption{Compute $F^{\infty}(x,y,z,y)$ \label{Finfty}}
\begin{algorithmic}[1]
\State Work at precision $\mathcal{P}$.
\State $n \gets 0$
\While{$|z - t| \leq 2^{-P - n - c_1}$}
	\State $n \gets n+1$
	\State $(x,y,z,t) \gets F(x,y,z,t)$
\EndWhile
\State $(x,y,z,t) \gets F(x,y,z,t)$
\State Return $\left( \left( \frac{x}{z} \right)^{2^{n+1}} \times z, z \right)$
\end{algorithmic}
\end{algorithm}

This gives an algorithm, Algorithm~\ref{Finfty}, to compute $F^{\infty}(x,y,z,t)$. According to \cite[Theorem~12]{Dupont}, if $n = \max( \log |\log| z_0/t_0||, 1) + \log(P+c_1)$, $a_n$ is an approximation of $\AGM\left(1, |\frac{z_0}{t_0}|\right)$ with relative precision $P$ bits. This proves that at the end of the algorithm, $n = O(\log P)$; in fact, we have more precisely $n \leq \log_2 P + C''$ with $C''$ a constant independent of $P$. Finally in the next subsection proves that one can take $\mathcal{P} = P + O(\log P)$, which means that this algorithm computes $F^{\infty}$ in $O(\mathcal{M}(P) \log P)$ bit operations.

\subsection{Loss of precision}
\label{section:Finftybitslost}

We use Theorem~\ref{theoremlossofprec} in order to evaluate the precision lost when computing $F^{\infty}(x,y,z,t)$. First note that the upper and lower bounds on the terms of the sequence allow us to write
\begin{eqnarray*}
\left(\frac{1}{\sqrt{|z_n|}} + \frac{1}{\sqrt{|t_n|}}\right) \left(\sqrt{|z_n|} + \sqrt{|t_n|}\right) & \leq & b/2 \\
\left(\frac{1}{\sqrt{|z_n|}} + \frac{1}{\sqrt{|t_n|}}\right) \left(\sqrt{|x_n|} + \sqrt{|y_n|}\right) & \leq & b \\
\left(\frac{1}{\sqrt{|x_n|}} + \frac{1}{\sqrt{|y_n|}}\right) \left(\sqrt{|z_n|} + \sqrt{|t_n|}\right) & \leq & b
\end{eqnarray*}
for some $b >1$; for instance, one can take $b = \max\left(1, 4 \sqrt{\frac{C}{\epsilon}}\right)$. We prove in Section~\ref{section:proveAllHypo} the existence of $\epsilon$ and $C$, and hence of $b$, for any values of theta we consider as arguments.

We first evaluate a bound on the error incurred when computing $F$ using Equation~\eqref{FKaratsuba}. Using those formulas allows us to get error bounds that are identical for $F_x$ and $F_y$, and $F_z$ and $F_t$. For simplicity, we assume that the error on $z$ and $t$ is the same, as well as the error on $x$ and $y$. This gives:
\begin{eqnarray*}
|\Re(F_x) - \Re(\tilde{F_x})| & \leq & \left(1 + k_z \left(\frac{1}{\sqrt{|z|}} + \frac{1}{\sqrt{|t|}}\right) \left(\sqrt{|x|} + \sqrt{|y|}\right) + k_x\left(\frac{1}{\sqrt{|x|}} + \frac{1}{\sqrt{|y|}}\right) \left(\sqrt{|z|} + \sqrt{|t|} \right) \right) 2^{-P} \\
|\Re(F_z) - \Re(\tilde{F_z})| & \leq & \left(1 + 2 k_z \left(\frac{1}{\sqrt{|z|}} + \frac{1}{\sqrt{|t|}} \right) \left(\sqrt{|z|} + \sqrt{|t|} \right) \right) 2^{-P}
\end{eqnarray*}
We thus get the following induction relations when looking at what happens when applying $F$ $n$ times in a row:
\begin{equation*}
k^{(n)}_x \leq 1 + b k^{(n-1)}_z + b k^{(n-1)}_x, \qquad k^{(n)}_z \leq 1 + b k^{(n-1)}_z
\end{equation*}
The last equation rewrites as $k^{(n)}_z + \frac{1}{b-1} \leq b\left(k^{(n-1)}_z + \frac{1}{b-1}\right)$, which gives $k^{(n)}_z \leq b^n\left(k_z+\frac{1}{b-1}\right)$. The induction for $x$ becomes $k^{(n)}_x \leq 1 + (k_z + \frac{1}{b-1}) b^n + b k^{(n-1)}_x$, which we solve:
\begin{equation*}
k^{(n)}_x \leq (1+b+...+b^n)k_x + nb^n \left(k_z + \frac{1}{b-1}\right) \leq b^n\left( nk_z + \frac{b+n}{b-1} \right)
\end{equation*}
For $b>1$, we have for $n$ large enough that $k^{(n)} \leq 2b^{2n}$, which ultimately means the number of bits lost when applying $F$ $n$ times in a row is bounded by $2n \log b + 1$.

Finally we need to find the number of bits lost in the computation of $\left(\frac{x_n}{z_{\infty}} \right)^{2^n}$. Call $E_k$ the error made after computing $k$ squarings in a row; we have the following recurrence relation:
\begin{equation*}
E_{k+1} \leq 2 + 4 E_k |x_n/z_{\infty}|^{2^k}
\end{equation*}
However, since $(\lambda_n)$ converges, $|\lambda_n| \leq \rho$ for some constant $\rho$; furthermore, for any $k \leq n$, one has $|x_n/z_{\infty}|^{2^k} \leq 1 + \frac{\rho}{z_{\infty}}$. Hence the recurrence becomes $E_{k+1} \leq 2 + 4 \left(1+ \frac{\rho}{z_{\infty}}\right) E_k$, which we solve to get
\begin{equation*}
E_n \leq 2 \frac{C'^{n+1}-1}{C'-1} \leq \frac{2}{C'-1} C'^{n+1}
\end{equation*}
with $C' = 4\left(1+ \frac{\rho}{z_{\infty}}\right)$. This means the number of bits lost after $n$ successive squarings is at the most $(n+1) \log C' + 1 - \log(C'-1)$.

Overall, if we write that the final value of $n$ in Algorithm~\ref{Finfty} is bounded by $\log_2 P + C''$, we have that the number of bits lost is bounded by
\begin{equation*}
(2 \log_2 b + \log C') (\log_2 P + C'') + \log C' + 2 - \log C'-1
\end{equation*}
which is $O(\log P)$.

\section{\texorpdfstring{Fast computation of $\theta$}{Fast computation of theta}}

We use a similar method as~\cite{Dupont}, that is to say finding a function $\mathfrak{F}$ such that
\begin{equation*}
\mathfrak{F}\left( \frac{\theta_{01}^2(z, \tau)}{\theta_{00}^2(z,\tau)}, \frac{\theta_{01}^2(0, \tau)}{\theta_{00}^2(0, \tau)} \right) = (z, \tau),
\end{equation*}
which can then be inverted using Newton's method. One can then compute $\theta(z,\tau)$ by, for instance, using Equation~\eqref{eq:OurAGMQuotients} and extracting a square root, determining the correct choice of sign by computing a low-precision (say, 10 bits) approximation of the value using the naive method; we use a different trick in our final algorithm (Algorithm~\ref{theUniformAlgo}). We build this function $\mathfrak{F}$ using $F^{\infty}$ as a building block.

\subsection{\texorpdfstring{Building $\mathfrak{F}$}{Building our function to invert}}

Just as with the algorithm for theta-constants, we use formulas derived from the action of $SL_2(\mathbb{Z})$ on the values of $\theta$ in order to get multiplicative factors depending on our parameters; this will allow us to build a function which computes $z, \tau$ from the values $\theta_i(z,\tau)$. We define the function $\mathfrak{F}$ as the result of Algorithm~\ref{mathfrakFalgo}.
\begin{algorithm}[H]
\caption{Compute $\mathfrak{F}\left(s,t \right)$ \label{mathfrakFalgo}}
\begin{algorithmic}[1]
\State $b\gets \sqrt{1 - t'^2}$
\Comment{Choose the root with positive real part \cite[Prop~2.9]{Cox} }
\State $a\gets \frac{1 - st}{b}$
\State $(x,y)\gets F^{\infty}\left(1, a, 1, b\right)$
\State $(q_1, q_2)\gets F^{\infty}\left(1, s, 1, t\right)$
\State Return $\left( \sqrt{\log\left( \frac{q_2 x}{q_1 y} \right) \times \frac{q_2/y}{-2\pi} }, i\frac{q_2}{y} \right)$, choosing the sign of the square root so that it has positive imaginary part.
\end{algorithmic}
\end{algorithm}

\begin{proposition} \label{propmathfrakF}
Let $\tau$ be such that $|\Re(\tau)| \leq 0.5$, $\Im(\tau) \geq \Value$ and $\Im\left(\frac{-1}{\tau}\right) \geq \Value$, and let $z$ be such that the conditions~\eqref{eq:conditionsz} are satisfied. Then
\begin{equation*}
\mathfrak{F}\left( \frac{\theta_{01}^2(z,\tau)}{\theta_{00}^2(z,\tau)}, \frac{\theta_{01}^2(0,\tau)}{\theta_{00}^2(0,\tau)} \right) = (z, \tau)
\end{equation*}
\end{proposition}
\begin{proof}
Equation~\eqref{eq:OurAGMQuotients} proves that $(q_1, q_2) = \left( \frac{1}{\theta_{00}(z,\tau)^2}, \frac{1}{\theta_{00}(0,\tau)^2} \right)$. Furthermore, using Jacobi's formula~\eqref{eq:Jacobi} and the equation defining the variety~\eqref{eq:Variety}, it is easy to see that $b = \frac{\theta_{10}(0,\tau)^2}{\theta_{00}(0,\tau)^2}$ and $a = \frac{\theta_{10}(z,\tau)^2}{\theta_{00}(z,\tau)^2}$.

The formulas in~\cite[Table~V, p.36]{Mumford1} give
\begin{equation*}
\left(\theta_{00}^2(z, \tau), \theta_{10}^2(z,\tau), \theta_{00}^2(0, \tau), \theta_{10}^2(0,\tau)\right) = \left(\lambda \theta_{00}^2\left(\frac{z}{\tau}, \frac{-1}{\tau}\right), \lambda \theta_{01}^2\left(\frac{z}{\tau}, \frac{-1}{\tau}\right), \mu \theta_{00}^2\left(0, \frac{-1}{\tau}\right), \mu \theta_{01}^2\left(0, \frac{-1}{\tau}\right) \right)
\end{equation*}
with $\lambda = \frac{e^{-2i \pi z^2/\tau}}{-i \tau}$, $\mu = \frac{1}{-i \tau}$. From the discussion in Section~\ref{section:moreArgReduction}, the conditions on $z, \tau$ allow us to apply Theorem~\ref{goodchoicethetas} to $\frac{z}{\tau}, \frac{-1}{\tau}$. This proves that
\begin{equation*}
F^{\infty}\left(\theta_{00}^2(z, \tau), \theta_{10}^2(z,\tau), \theta_{00}^2(0, \tau), \theta_{10}^2(0,\tau)\right) = \left( \frac{e^{-2i \pi z^2/\tau}}{-i \tau}, \frac{1}{-i \tau} \right),
\end{equation*}
and by homogeneity, $(x,y) = \left( \frac{e^{-2i \pi z^2/\tau}}{-i \tau \theta_{00}(z,\tau)^2}, \frac{1}{-i \tau \theta_{00}(0,\tau)^2} \right)$.
\end{proof}

This means that, starting from the knowledge of $z$ and $\tau$ with precision $P$ and a low-precision approximation of the quotients $\frac{\theta_{01}(z, \tau)}{\theta_{00}(z, \tau)}$ and $\frac{\theta_{01}(0, \tau)}{\theta_{00}(0, \tau)}$, one can compute those quotients with precision $P$ using Newton's method. This is Algorithm~\ref{fastthetas}.
\begin{algorithm}
    \caption{Compute $\theta_{00}^2(z,\tau), \theta_{01}^2(z,\tau), \theta_{00}^2(0,\tau), \theta_{01}^2(0,\tau)$ with precision $P$. \\ Input: $(z, \tau)$ with absolute precision $P$.}
    \label{fastthetas}
\begin{algorithmic}[1]
\State Compute $\theta_{00,01}^2(z,\tau)$, $\theta_{00,01}^2(0,\tau)$ with absolute precision $P_0$ using Algorithm~\ref{naive}.
\State $s \gets \frac{\theta_{01}(z,\tau)^2}{\theta_{00}(z,\tau)^2}$, $t \gets \frac{\theta_{01}(0,\tau)^2}{\theta_{00}(0,\tau)^2}$
\State $p \gets P_0$
\While{$p \leq \mathcal{P}'$}
  \State $p \gets 2p$
  \State Compute $a_{11} = \frac{\partial \mathfrak{F}_x}{\partial x}(s,t)$, $a_{22} = \frac{\partial \mathfrak{F}_y}{\partial y}(s,t)$, $a_{12} = \frac{\partial \mathfrak{F}_x}{\partial y}(s,t)$ with precision $p$.
  \State $(s,t) \gets (s,t) - (\mathfrak{F}(s,t)-(z,\tau)) \begin{pmatrix} a_{11} & a_{12} \\ 0 & a_{22}  \end{pmatrix}^{-1}$ \label{fastthetas:inverseJacobian}
\EndWhile
\State $(a,b)\gets F^{\infty}(1,s,1,t)$
\State $(a,b)\gets (1/a, 1/b), \qquad (s,t) = (sa, tb)$
\State Return $(a,s,b,t)$.
\end{algorithmic}
\end{algorithm}

We make a few remarks:
\begin{itemize}
\item Much in the same way as~\cite{ThomeEnge}, we find it preferable to use finite differences to compute the coefficients $a_{11}, a_{21}, a_{22}$ of the Jacobian, as it does not require the computation of the derivative of $\mathfrak{F}$, which could be tedious.
\item The value of $P_0$ has to be large enough that Newton's method converges. We note that, in general, a lower bound on $P_0$ may depend on the arguments; for instance,~\cite{Dupont} experimentally finds $4.53\Im(\tau)$ to be a suitable lower bound for $P_0$ when computing theta-constants. However, we outline in the next section a better algorithm which only uses the present algorithm for $z, \tau$ within a compact set; hence, $P_0$ can be chosen to be a constant, and we use in practice $P_0 = 30000$.
\end{itemize}

We do not provide a full analysis for this algorithm: we outline in the next section a better algorithm, which uses this algorithm as a subroutine, and we will provide a full analysis at that time. It is enough to say that the computation of $F^{\infty}$ at precision $p$ is done in time $O(\mathcal{M}(p) \log p)$ using Algorithm~\ref{Finfty}; however, this running time depends on $z, \tau$, since it depends on the bounds $C, \epsilon$ that one can write for $|x_n|, |y_n|, |z_n|, |t_n|$. Hence, the cost of evaluating $\mathfrak{F}$ at precision $p$ is $O(\mathcal{M}(p) \log p)$ bit operations, and the fact that we double the working precision at every step means that the algorithm is as costly as the last iteration. Furthermore, one should choose $\mathcal{P}'$ so that the final result is accurate with absolute precision $P$. This means compensating the loss of absolute precision incurred during the computation of $\mathfrak{F}$; in general, this only depends on $\Im(\tau)$ and linearly in $\log p$. Furthermore, we have the following proposition:
\begin{proposition} \label{newtondelta}
Let $x$ and $y$ be approximations of $a$ and $b$ with absolute precision $N$, and $x', y'$ the result of Step~\ref{fastthetas:inverseJacobian} in Algorithm~\ref{fastthetas} when using finite differences to approximate the Jacobian matrix. Then $x', y'$ are approximations of $a,b$ with precision $2P_0 - \delta$.
\end{proposition}
Its proof can be adapted from the proof of~\cite[Theorem~12]{ThomeEnge}. In practice, we found $\delta = 4$ to be enough. Determining the number of bits lost at each step can be done in the same way as~\cite[p.~19]{ThomeEnge}: if $s^{(n-1)}$ and $s^{(n-2)}$ agree to $k$ bits, and $s^{(n)}$ and $s^{(n-1)}$ agree to $k'$ bits, the number of bits lost can be computed as $2k-k'$. In the end, working at precision $\mathcal{P}' = P + c \log P + d$, with $c, d$ independent of $P$ but functions of $z, \tau$, is enough to compensate all precision losses; this proves that the running time of this algorithm is asymptotically $O(\mathcal{M}(P) \log P)$.

\subsection{\texorpdfstring{Computing $\theta(z, \tau)$ in uniform quasi-optimal time}{Computing theta(z, tau) in uniform quasi-optimal time}}
\label{section:uniformalgorithm}

We now show an algorithm with uniform (i.e. independent in $z$ and $\tau$) quasi-optimal complexity that computes $\theta(z,\tau)$ for any $(z,\tau)$ satisfying conditions~\eqref{eq:conditions}. We use the same strategy as~\cite{Dupont}; namely, we use the naive algorithm when $\Im(\tau)$ is large; and for smaller values of $\Im(\tau)$, we put $\tau' = \frac{\tau}{2^s}$ so that $\tau'$ is within a compact set, then use Algorithm~\ref{fastthetas}, which complexity will be uniform since its arguments belong to a compact set. However we also need to divide $z$ by a power of 2 so that it also belongs to a compact set, and so that $(z', \tau')$ satisfies conditions~\eqref{eq:conditions} and~\eqref{eq:conditionsz}. Once $\theta\left(\frac{z}{2^t}, \frac{\tau}{2^s}\right)$ has been computed by the previous algorithm, we alternate between using Equation~\eqref{eq:Tauduplication00} to double the second argument and Equation~\eqref{eq:Zduplication} to double the first argument, until finally recovering $\theta(z,\tau)$. This is Algorithm~\ref{theUniformAlgo}.

\begin{algorithm}
\caption{Compute $\theta(z, \tau)$ for $\tau \in \mathcal{F}$ and $z$ reduced} \label{theUniformAlgo}
\begin{algorithmic}[1]
\If{$P \leq 25 \Im(\tau)$}
  \State Compute $\theta_{00, 01, 10}(z, \tau), \theta_{00, 01, 10}(0, \tau)$ with precision $P$ using the naive method (Algorithm~\ref{naive} + Section~\ref{section:theta10naive}). \label{uniformalgo:naive}
\Else
  \State Take $s \in \mathbb{N}$ such that $1 \leq |\tau|/2^s < 2$
  \State Put $\tau_1 = \frac{\tau}{2^s}$ and $z_1 = \frac{z}{2^s}$, so that $\Im(z_1) \leq \Im(\tau_1)/2$.
  \State Put $z_2 = z_1/4$ and $\tau_2 = \tau_1/2$.
  \State Compute approximations of absolute precision $\mathcal{P}$ of $\theta_{00}^2(z_2, \tau_2), \theta_{01}^2(z_2, \tau_2)$, $\theta_{00}^2(0, \tau_2)$, and $\theta_{01}^2(0, \tau_2)$ using Algorithm~\ref{fastthetas}. \label{uniformalgo:subroutine}
  \State Compute $\theta_{00}^2(z_2, \tau_1), \theta_{01}^2(z_2, \tau_1), \theta_{00}^2(0, \tau_1), \theta_{01}^2(0, \tau_1)$ using Equation~\eqref{eq:Tauduplication00}, and $\theta_{10}^2(z_2, \tau_1)$ using Equation~\eqref{eq:Tauduplication10} and $\theta_{10}^2(0,\tau_1)$ using its equivalent in $z=0$. \label{uniformalgo:trickfor10}
  \State Compute $\theta_{00,01,10}(0,\tau_1)$. \label{uniformalgo:squareroot}
  \State Compute $\theta_{00,01}(z_1/2, \tau_1)$ using Equation~\eqref{eq:Zduplication}. \label{uniformalgo:zDoubling}
  \For{i = 1 .. s} \label{uniformalgo:beginfor}
    \State Compute $\theta_{00}^2(0, 2^i\tau_1), \theta_{01}^2(0, 2^i\tau_1)$ using the AGM.
  	\State Compute $\theta_{00}^2(2^{i-2}z_1, 2^i\tau_1), \theta_{01}^2(2^{i-2}z_1, 2^i\tau_1)$ using Equation~\eqref{eq:Tauduplication00}. \label{uniformalgo:taudoubling}
  	\State If $i=s$, compute also $\theta_{10}^2(0, 2^i\tau_1)$ using the equivalent of Equation~\eqref{eq:Tauduplication10} in $z=0$, then $\theta_{10}(0,2^i\tau_1)$ by taking the square root. \label{uniformalgo:squareroot10}
  	\State Compute $\theta_{10}^2(2^{i-2}z_1, 2^i\tau_1)$ using Equation~\eqref{eq:Tauduplication10}. \label{uniformalgo:taudoubling10}
    \State Compute $\theta_{00,01}(0,2^i\tau_1)$. \label{uniformalgo:squarerootInLoop}
  	\State Compute $\theta_{00}(2^{i-1} z_1, 2^i \tau_1), \theta_{01}(2^{i-1} z_1, 2^i \tau_1)$ using Equation~\eqref{eq:Zduplication}. \label{uniformalgo:duplicationInLoop}
  \EndFor \label{uniformalgo:endfor}
  \State Compute $\theta_{10}^2(2^{s-1} z_1, 2^s \tau_1)$ using Equation~\eqref{eq:Variety}. \label{uniformalgo:recover10}
  \State Compute $\theta_{00,01,10}(z, \tau)$ using Equation~\eqref{eq:Zduplication}. \label{uniformalgo:finalDuplication}
\EndIf
\end{algorithmic}
\end{algorithm}
A few notes on this algorithm:
\begin{itemize}
\item We note that, at several steps of the algorithm (e.g. Steps~\ref{uniformalgo:squareroot}, \ref{uniformalgo:squareroot10}, \ref{uniformalgo:squarerootInLoop}) we need to compute theta-constants from their square. The correct choice of signs is given by the proof of Theorem~\ref{goodchoicethetas}, which shows that $\Re(\theta_{00}(0,\tau)) \geq 0$ and $\Re(\theta_{01}(0,\tau)) \geq 0$; and furthermore, since $\Re(q^{1/4}) \geq |q|^{1/4} \cos(\pi/8)$, we also have $\Re(\theta_{10}(0,\tau)) \geq 0$.
\item Taking $\tau_2 = \tau_1/2$ allows us to use Equation~\eqref{eq:Tauduplication10} in step~\ref{uniformalgo:trickfor10} instead of Equation~\eqref{eq:Jacobi} and Equation~\eqref{eq:Variety}, which is more efficient and loses fewer bits.
\item The knowledge of $\theta_{10}^2(2^{i-2} z_1, 2^i \tau)$ is enough for the $z$-duplication formulas of step~\ref{uniformalgo:duplicationInLoop}, and it can be computed directly from $\theta_{00}$ and $\theta_{01}$ using Equation~\eqref{eq:Tauduplication10}.
\item Computing $\theta_{11}(z,\tau)$ is also possible; one should use a partial summation if $P \leq 25 \Im(\tau)$. In the other case, since all the $z$-duplication formulas for $\theta_{11}(z,\tau)$ involve a division by $\theta_{10}(0,\tau)$~\cite[p.22]{Mumford1}, it is just as efficient to simply use Equation~\eqref{eq:recoverTheta11} after Step~\ref{uniformalgo:finalDuplication}, then extract the square root. The square root extraction loses $O(\Im(\tau)) = O(P)$ bits, and this also gives a quasi-optimal algorithm.
\end{itemize}

\subsection{Proving the correctness of the algorithm}

This section is devoted to proving the following theorem:
\begin{theorem} \label{uniformtheorem}
For any $\tau, z$ satisfying conditions~\eqref{eq:conditions}, Algorithm~\ref{theUniformAlgo} with $\mathcal{P} = 2P$ computes $\theta_{00,01,10}(z,\tau), \theta_{00,01,10}(0,\tau)$ with absolute precision $P$ in $O(\mathcal{M}(P) \log P)$ bit operations.
\end{theorem}
As we discussed in Section 2, this also gives an algorithm that computes $\theta(z,\tau)$ for any $(z, \tau) \in \mathbb{C} \times \mathcal{H}$; one simply needs to reduce $\tau$ in $\tau' \in \mathcal{F}$, then reduce $z$ in $z'$, and deduce $\theta(z,\tau)$ from $\theta(z', \tau')$ using Equations~\eqref{eq:quasiperiodicitytheta} and \eqref{eq:sl2andthetavalues}. This causes a loss of absolute precision which depends on $z$ and $\tau$, and this algorithm is no longer uniform.

We need to perform an analysis of the number of bits lost by the algorithm; once again, we use Theorem~\ref{theoremlossofprec}. For each step, we proceed as follows: assuming the error on all the quantities is bounded by $k$, determine a factor $x$ such that the error on the quantities we get after the computation is bounded by $xk$, then declare the number of bits lost in this step to be $\log x$; this gives a very loose upper bound, but simplifies the process.

Finally, we also need to prove that the hypotheses made in Sections~\ref{section:convergence} and~\ref{section:Finftybitslost} are verified in Step~\ref{uniformalgo:subroutine} of the algorithm. This is necessary to prove that the sequence $(\lambda_n)$ we consider is quadratically convergent, and that the number of bits lost is only $O(\log P)$. We prove this in Section~\ref{section:proveAllHypo}, which then completes the proof that the running time is indeed uniform and quasi-optimal.

\subsubsection{Naive algorithm}
\label{section:subsub1}

As we showed in Theorem~\ref{naiveanalysis}, the number of bits lost when using the naive algorithm is $\log B + 7$, although this constant could be made even smaller when taking into account that $P \leq 25 \Im(\tau)$. Furthermore, $\sqrt{\frac{P}{\Im(\tau)}} \leq 25$, which means the running time of this step is asymptotically dominated by the cost of the computation of $\pi$, $q$ and $w$ with precision $\mathcal{P} = P + \log B + 7$, which takes $O(\mathcal{M}(P) \log P)$ bit operations.

\subsubsection{Square root extraction}

Steps~\ref{uniformalgo:squareroot}, \ref{uniformalgo:squarerootInLoop} and~\ref{uniformalgo:squareroot10} require extracting square roots, which multiply the error by $\frac{1}{\sqrt{|z|}}$. We prove in the next subsection that $|\theta_{00,01}(0, 2^i\tau_1)| \geq 0.859$ for $i = [1 \ldots s]$. Hence, each extraction of square root loses at most 4 bits: step~\ref{uniformalgo:squareroot} loses 4 bits, and step~\ref{uniformalgo:squarerootInLoop} loses $4s \leq 4 \log P$ bits.

Step~\ref{uniformalgo:squareroot10} loses more bits since $\theta_{10}(0,\tau)$ is smaller; indeed, $|\theta_{10}(0,\tau)| \sim |q|^{1/4}$. This means the number of bits lost during this step is bounded by $\frac{\log |q|}{8} = \frac{\pi}{8} \log_2 e \Im(\tau)$.

\subsubsection{Duplication formulas and finishing the proof of correctness}
\label{section:subsub2}

The algorithm uses both $\tau$-duplication formulas and $z$-duplication formulas, and we need to analyse how many bits are lost for each application of those formulas.

The $\tau$-duplication formulas are nothing more than applying $F$ to $\theta_{00,01}^2(z, \tau)$ and $\theta_{00,01}^2(0, \tau)$. However, the analysis here is simpler than in section~\ref{section:Finftybitslost}, because we do not need to compute the square roots of $\theta_{00,01}(z, \tau)$, since they are directly given by step~\ref{uniformalgo:duplicationInLoop}. Hence we just need to account for the error of the additions, subtractions and multiplications in Equation~\eqref{FKaratsuba}; since all the quantities are bounded, this means each step loses a constant number $g$ of bits (our analysis shows that $g \leq 10.48$). In the end, the $\tau$-duplication formulas account for the loss of $g \times s \leq g \log P$ bits of precision.

As for the $z$-duplication formulas, we need to perform several analyses. Looking at Equation~\eqref{eq:Zduplication}, one needs to evaluate the fourth power of theta functions, then add them; then evaluate the third power of theta constants, then perform a division. Computing the error using the formulas from \ref{theoremlossofprec} is rather straightforward when one has bounds on those quantities, which are given by the following theorem:
\begin{theorem} \label{boundonthetas}
Assume $\Im(\tau) > \sqrt{3}/2$. Then
\begin{equation*}
0.859  \leq |\theta_{00,01}(0,\tau)| \leq 1.141,  \qquad |\theta_{10}(0,\tau)| \leq 1.018
\end{equation*}
We also have:
\begin{itemize}
\item Suppose that $0 \leq \Im(z) \leq \frac{\Im(\tau)}{8}$, as in Steps~\ref{uniformalgo:zDoubling} and~\ref{uniformalgo:duplicationInLoop}. Then $|w|^{-2n} \leq e^{n \pi \Im(\tau)/4}$ and
\begin{equation*}
0.8038 \leq |\theta_{00,01}(z,\tau)| \leq 1.1962 \qquad |\theta_{10}(z,\tau)| \leq 1.228
\end{equation*}
\item Suppose that $0 \leq \Im(z) \leq \frac{\Im(\tau)}{4}$, as in Steps~\ref{uniformalgo:finalDuplication}. Then $|w|^{-2n} \leq e^{n \pi \Im(\tau)/2}$ and
\begin{equation*}
0.6772 \leq |\theta_{00,01}(z,\tau)| \leq 1.3228 \qquad |\theta_{10}(z,\tau)| \leq 1.543
\end{equation*}
\end{itemize}
\end{theorem}
\begin{proof}
The bounds on the theta-constants come from~\cite[p.~5]{Dupont}, which proves $|\theta_{00,01}(0,\tau)-1| \leq \frac{2|q|}{1-|q|}$. The techniques are the same as the proof of Lemma~\ref{realPartsOfQuotients} or Theorem~\ref{naiveboundtheta}. This gives in the first case
\begin{eqnarray*}
|\theta_{00,01}(z,\tau) -1| & \leq & |q|^{3/4} + |q| + |q|^{7/2} + |q|^4 + |q|^{8.25} + |q|^9 + \frac{|q|^{15}}{1-|q|} \\
& \leq & 0.1962 \text{ \qquad since $\Im(\tau) \geq \sqrt{3}/2$ } \\
|\theta_{10}(z,\tau) - q^{1/4}(w + w^{-1})| & \leq & \frac{q^{15/8}}{1-q^{3/8}} \leq 0.009
\end{eqnarray*}
so $|\theta_{10}(z,\tau)| \leq |q|^{1/4}(|w|+|w|^{-1}) + 0.009 \leq 1.228$. In the second case:
\begin{eqnarray*}
|\theta_{00,01}(z,\tau) -1| & \leq & |q|^{1/2} + |q| + \frac{|q|^{3}}{1-|q|} \leq  0.3228 \\
|\theta_{10}(z,\tau) - q^{1/4}(w + w^{-1})| & \leq & |q|^{5/4} + |q|^{9/4} +  ... \leq \frac{q^{5/4}}{1-q} \leq 0.0357
\end{eqnarray*}
\end{proof}
Combining these bounds with formulas from Theorem~\ref{theoremlossofprec} gives the following bounds:
\begin{eqnarray*}
\text{error}(\theta_{00}(2^i z_1, 2^{i+1} \tau_1)) & \leq & ( 20050.518 + 1818.032 k_{\theta_{01}^z} + 1966.823 k_{\theta_{10}^z} + 33516 k_{\theta_{00}^0} ) 2^{-P} \\
\text{error}(\theta_{01}(2^{i} z_1, 2^{i+1} \tau_1)) & \leq & ( 20050.518 + 1818.032 k_{\theta_{00}^z} + 1966.823 k_{\theta_{10}^z} + 33516 k_{\theta_{01}^0} ) 2^{-P} \\
\end{eqnarray*}
which means losing at most 16 more bits of precision.

Step~\ref{uniformalgo:recover10} causes the loss of a greater number of bits. We use Equation~\eqref{eq:Variety} instead of the third $z$-duplication formula, because dividing by $\theta_{10}(0,\tau)^2$ loses less bits than dividing by $\theta_{10}(0,\tau)^3$, and we only need the knowledge of $\theta_{10}^2(2^{s-1}z, 2^s\tau)$ for the next step anyway. This amounts to computing:
\begin{equation*}
\theta_{10}^2(z,\tau) = \frac{\theta_{00}^2(z,\tau) \theta_{00}^2(0,\tau) - \theta_{01}^2(z,\tau) \theta_{01}^2(0,\tau)}{\theta_{10}^2(0,\tau)}
\end{equation*}
Computing the numerator multiplies the error by a factor at most 60, and the norm of this numerator is bounded by 4.557; we then get from Theorem~\ref{theoremlossofprec} that the error is bounded by $\frac{m}{|\theta_{10}(0,\tau)|^8} \sim m |q|^{-2}$, with $m \leq 1600$. In the end, we lose at most $2 \pi \log_2 e \Im(\tau) + 11$ bits.

Finally, we also lose a great number of bits during the last application of the $z$-duplication formulas in step~\ref{uniformalgo:finalDuplication}, since the formula for $\theta_{10}(z,\tau)$ requires dividing by $\theta_{10}(0,\tau)^3$. The error is thus multiplied by $|q|^{-3}$ up to a constant factor; this means a loss of $3\pi~\log_2 e \Im(\tau)$ bits, plus a constant.

In the end, we see that the number of lost bits is bounded by $(2 \pi + \pi/8 + 3\pi)\Im(\tau)~\log_2 e + c \log P + d$; given that $P \geq 25 \Im(\tau)$, and that $5.125 \pi \log_2 e \leq 23.3$, the number of bits lost is thus less than $P$. This means that $\mathcal{P} = 1.93 P + c \log P + d \leq 2P$ is enough to get a result which is accurate to absolute precision $P$; this also means that we indeed never have an error $k$ bigger than $2^{(2P)/2}$, which is necessary to apply Theorem~\ref{theoremlossofprec}.

\subsubsection{Proof of quadratic convergence and quasi-optimal running time}
\label{section:proveAllHypo}

It remains to prove that the complexity is the right one. If $P \geq 25 \Im(\tau)$, $\log_2 P > \log_2 \Im(\tau) + 4.7$, which means $s \leq \log P$ and the cost of Steps~\ref{uniformalgo:beginfor} to \ref{uniformalgo:endfor} is $O(\mathcal{M}(\mathcal{P}) \log P)$. We verify that conditions~\eqref{eq:conditions} and~\eqref{eq:conditionsz} hold:
\begin{eqnarray*}
|\Re(\tau_2)| \leq 1/2^{s+2} \leq 1/4, && \Value \leq \frac{\sqrt{3}}{4} \leq \Im(\tau_2) \leq 1, \\
|\Re(z_2)| \leq 1/2^{s+3} \leq 1/8, && 0 \leq \Im(z_2) \leq \frac{\Im(\tau_2)}{4}.
\end{eqnarray*}
This means the choices of signs are always good, and hence our result is indeed (squares of) theta-functions and theta-constants.

We also need prove that there is a $C > 1$ such that, for all $z_2, \tau_2$ that we consider,
\begin{equation*}
\frac{\theta_{01}^2(0,\tau_2)}{\theta_{00}^2(0,\tau_2)} \leq C, \quad \frac{\theta_{10}^2(0,\tau_2)}{\theta_{00}^2(0,\tau_2)} \leq C, \qquad \frac{\theta_{01}^2(z_2,\tau_2)}{\theta_{00}^2(z_2,\tau_2)} \leq C, \quad \frac{\theta_{10}^2(z_2,\tau_2)}{\theta_{00}^2(z_2,\tau_2)} \leq C.
\end{equation*}
This is a direct consequence of the fact that $z_2, \tau_2$ are within a compact set that does not contain any zero of $\theta(z,\tau)$; hence one can write (non-zero) lower and upper bounds for any of the values of theta. If one wants to be a bit more precise, using the same reasoning as in Theorem~\ref{boundonthetas}, we have for $\sqrt{3}/4 \leq \Im(\tau) \leq 1$:
\begin{equation*}
|\theta_{00,01}(z,\tau) - 1| \leq |q|^{1/2}+|q|+\frac{|q|^3}{1-|q|} \leq 0.7859, \qquad |\theta_{10}(z,\tau)| \leq 1 + |q|^{1/4} + \frac{|q|^{5/4}}{1-|q|} \leq 1.958.
\end{equation*}
This gives $C \leq 83.64$ and $\epsilon \geq \frac{0.042^2}{1.7859^2} \simeq \frac{1}{1808}$. Furthermore, with a careful analysis, one can prove that $c_1 = 55$ is enough in Theorem~\ref{thmconvergence}.

In any case, this proves that $(\lambda_n)$ is quadratically convergent. We note that the fact that $z_2, \tau_2$ are within a compact shows that the constants $b_1, b_2, b_3$ exist and are independent of $z, \tau$. This makes the running time of Step~\ref{uniformalgo:subroutine} only dependent in $P$, which was the point of the uniform algorithm. In particular, the number of bits lost during the computation of $F^{\infty}$ or in $\mathfrak{F}$ can be written as $c_1 \log P + c_2$, with $c_1, c_2$ constants independent in $z, \tau$. Hence, the number of bits that are lost in the whole of Step~\ref{uniformalgo:subroutine} is
\begin{equation*}
\sum_{i=1}^n \delta + h + \log(p/2^i) \leq G \log P + H
\end{equation*}
since the number $n$ of steps in Newton's method is $O(\log P)$.

This means the computations in step~\ref{uniformalgo:subroutine} should be carried out at precision $\mathcal{P}' = \mathcal{P} + G \log P + H$, so that the result is accurate with $\mathcal{P}$ bits. This gives a running time of $O(\mathcal{M}(P) \log P)$, independently of $z$ and $\tau$. All the other steps cost no more than $O(\mathcal{M}(\mathcal{P}))$ bit operations. Given the formula for $\mathcal{P}$ in subsections~\ref{section:subsub1} to~\ref{section:subsub2}, this indeed gives us a running time of $O(\mathcal{M}(P) \log P)$.

\section{Implementation}

An implementation using the GNU MPC library~\cite{MPC} for arithmetic on multiprecision complex numbers was developed; we compared our algorithm to our own implementation of Algorithm~\ref{naive} using MPC\footnote{The naive algorithm which only computes $\theta(z,\tau)$ is only 5\% faster; furthermore since Algorithm~\ref{theUniformAlgo} computes all 4 values, it is fair to compare it to Algorithm~\ref{naive}.}. The code is distributed under the GNU General public license, version 3 or any later version (GPLv3+); it is available at the address
\begin{center}
\url{http://www.hlabrande.fr/pubs/fastthetas.tar.gz}
\end{center}
We compared those implementations to MAGMA's implementation of the computation of $\theta(z,\tau)$ (function \verb?Theta?). Each of those implementations computed $\theta(z, \tau)$ for $z~=~0.123456789~+~0.123456789i$ and $\tau = 0.23456789 + 1.23456789i$ at different precisions; the computations took place on a computer with an Intel Core i5-4570 processor. The results are presented in Figure~\ref{benches} and Table~\ref{timings}.

Our figures show that our algorithm outperforms Magma even for computations at relatively low (i.e. 1000 digits) precision\footnote{This is even though Magma only returns $\theta(z,\tau)$, when our algorithm returns 4 values.}, and the naive algorithm for more than 325000 digits of precision. Hence, a combined algorithm which uses the naive method for precisions smaller than 325000 digits, and our method for larger precision, will yield the best algorithm, and outperform Magma in all cases, as shown in Table~\ref{timings}.

\bibliographystyle{abbrv}
\bibliography{biblioarticle}
\nocite{*}

\newpage

\begin{figure}
\caption{Timing results \label{benches}}
\begin{center}
\begin{tikzpicture}[scale=1]

    \begin{axis}[
            xmode=log,
        xlabel=Base 10 precision,
            xmin=500,
            ymode=log,
        ylabel=Time (s),
            ymax=10000000,
        ymin=0,
        ymajorgrids,
    ]

    \addplot [color=cyan] plot coordinates {
(500, 0.004000) 
(1000, 0.008000) 
(1500, 0.02000) 
(2000, 0.03120) 
(3000, 0.06000) 
(4000, 0.09200) 
(5000, 0.1360) 
(6000, 0.1840) 
(7000, 0.2360) 
(8000, 0.2984) 
(9000, 0.3480) 
(10000, 0.4160) 
(11000, 0.4792) 
(12000, 0.5480) 
(13000, 0.6120) 
(14000, 0.6992) 
(15000, 0.7824) 
(16000, 0.8680) 
(17000, 0.9448) 
(18000, 0.9960) 
(19000, 1.072) 
(20000, 1.163) 
(24000, 1.547) 
(28000, 1.956) 
(32000, 2.399) 
(36000, 2.869) 
(40000, 3.3224) 
(45000, 3.9896) 
(50000, 4.5992) 
(55000, 5.4440) 
(60000, 6.1568) 
(64000, 6.7784) 
(70000, 7.8416) 
(80000, 9.3168) 
(90000, 11.089) 
(100000, 12.686) 
(128000, 18.320) 
(256000, 45.5352)
(280000, 53.0368)
(300000, 55.6304)
(325000, 62.7360)
(350000, 70.2424)
(375000, 77.2008)
(400000, 82.3528)
(437000, 89.0496)
(474000, 104.385)
(512000, 111.782) 
(1024000, 263.698) 
(1536000, 432.8576)
(2048000, 625.3864)
(3072000, 996.8273)
(4096000, 1468.144)
    };    

    \addplot[color=black] plot coordinates {
(500, 0) 
(1000, 0) 
(1500, 0.004000) 
(2000, 0.008000) 
(3000, 0.02000) 
(4000, 0.03200) 
(5000, 0.04800) 
(6000, 0.07200) 
(7000, 0.09200) 
(8000, 0.1120) 
(9000, 0.1400) 
(10000, 0.1720) 
(11000, 0.2024) 
(12000, 0.2400) 
(13000, 0.2648) 
(14000, 0.3040) 
(15000, 0.3424) 
(16000, 0.3840) 
(17000, 0.4312) 
(18000, 0.4800) 
(19000, 0.5320) 
(20000, 0.5920) 
(24000, 0.8400) 
(28000, 1.044) 
(32000, 1.347) 
(36000, 1.650) 
(40000, 2.014) 
(45000, 2.502) 
(50000, 3.022) 
(55000, 3.4744) 
(60000, 4.1016) 
(64000, 4.5984) 
(70000, 5.3560) 
(80000, 6.7616) 
(90000, 8.4113) 
(100000, 9.9113) 
(128000, 15.285) 
(256000, 41.5552) 
(280000, 54.2864)
(300000, 50.7944)
(325000, 63.8952)
(350000, 82.3416)
(375000, 79.1400)
(400000, 89.2224)
(437000, 94.8496)
(474000, 120.385)
(512000, 129.880) 
(1024000, 390.3488) 
(1536000, 737.1728)
(2048000, 1275.351)
(3072000, 2243.728)
(4096000, 3921.2808)
    };
    
        \addplot[color=red] plot coordinates {
(500, 0.002000) 
(1000, 0.008000) 
(1500, 0.01800) 
(2000, 0.03400) 
(3000, 0.08800) 
(4000, 0.1740) 
(5000, 0.2960) 
(6000, 0.4580) 
(7000, 0.6540) 
(8000, 0.8719) 
(9000, 1.134) 
(10000, 1.456) 
(11000, 1.840) 
(12000, 2.248) 
(13000, 2.696) 
(14000, 3.2000) 
(15000, 3.7520) 
(16000, 4.3580) 
(17000, 5.0200) 
(18000, 5.7400) 
(19000, 6.5040) 
(20000, 7.3940) 
(24000, 11.522) 
(28000, 16.574) 
(32000, 22.702) 
(36000, 29.688) 
(40000, 38.1820) 
(45000, 50.3820) 
(50000, 64.7540) 
(55000, 81.0940) 
(60000, 99.4299) 
(64000, 116.442) 
(70000, 143.968) 
(80000, 198.436) 
(90000, 264.344) 
(100000, 340.9900) 
(128000, 606.0680) 
    };

     \legend{Our algo ($\mathcal{P}_0$ = 9000 digits)\\Naive\\Magma\\}
    \end{axis}

\end{tikzpicture}
\end{center}
\end{figure}
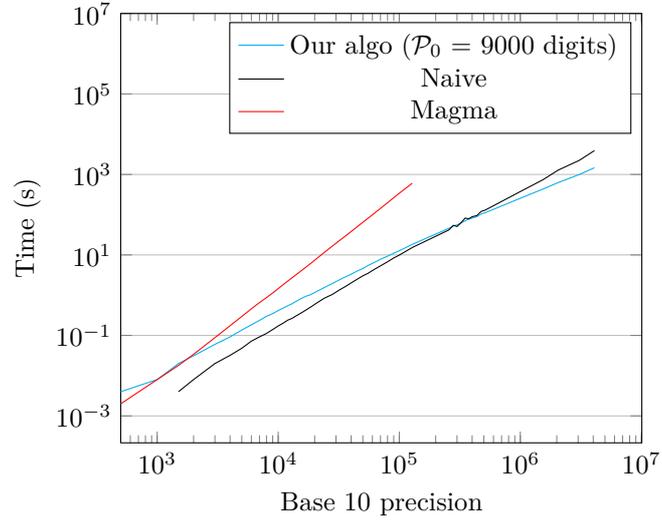

\begin{table}
\begin{center}
\begin{tabular}{c|c|c|c}
Prec (digits) & This work & Naive & Magma \\
\hline
4000 & 0.092 & 0.032 & 0.1740 \\
8000 & 0.298 & 0.112 & 0.8719 \\
16000 & 0.868 & 0.384 & 4.358 \\
32000 & 2.399 & 1.347 & 22.70 \\
64000 & 6.778 & 4.598 & 116.4 \\
128000 & 18.32 & 15.29 & 606.1 \\
256000 & 45.54 & 41.56 &   \\
325000 & 62.74 & 63.90 & \\
512000 & 111.78 & 129.8 &  \\
1024000 & 263.7 & 390.3 & \\
2048000 & 625.4 & 1275 & \\
4096000 & 1468 & 3921 & \\
\hline
\end{tabular}
\end{center}
\caption{\label{timings} Times (in s) of different methods}
\end{table}

\end{document}